\documentclass[10pt,a4paper]{article}
\usepackage{graphicx,mathrsfs,amssymb,color,dsfont}
\usepackage[utf8]{inputenc}
\usepackage{amsmath,bm,bbm}
\usepackage[colorlinks=true,linkcolor=blue,citecolor=blue,urlcolor=blue,breaklinks]{hyperref}
\usepackage{cleveref}
\usepackage{amssymb,amsthm,amsfonts}
\usepackage{xcolor,caption,float}
\usepackage{fixltx2e}

\usepackage{array}
\usepackage[all]{xy}
\usepackage{tikz}
\usetikzlibrary{arrows,shapes,trees} 
\usetikzlibrary{positioning}
\usepackage{multirow}

\newtheorem{theorem}{Theorem}[section]
\newtheorem{proposition}[theorem]{Proposition}
\newtheorem{lemma}[theorem]{Lemma}

\theoremstyle{definition}

\newtheorem{definition}[theorem]{Definition}
\newtheorem{remark}[theorem]{Remark}
\newtheorem{assumption}{Assumption}[section]
\numberwithin{equation}{section}

\def\Xint#1{\mathchoice
	{\XXint\displaystyle\textstyle{#1}}%
	{\XXint\textstyle\scriptstyle{#1}}%
	{\XXint\scriptstyle\scriptscriptstyle{#1}}%
	{\XXint\scriptscriptstyle\scriptscriptstyle{#1}}%
	\!\int}
\def\XXint#1#2#3{{\setbox0=\hbox{$#1{#2#3}{\int}$ }
		\vcenter{\hbox{$#2#3$ }}\kern-.6\wd0}}

\def\dashint{\Xint-}

\def\:{\colon}
\def\d{\,\mathrm{d}}

\def\R{\mathbb{R}}

\author{\textsc{José A.~Cañizo}\\
  Departamento de Matem\'atica Aplicada\\
  Universidad de Granada
  \\18071 Granada, Spain
  \\\texttt{\href{mailto:canizo@ugr.es}{canizo@ugr.es}}
  \and
  \textsc{José A.~Carrillo}\\Department of Mathematics\\Imperial College
  London\\London SW7 2AZ, United Kingdom
  \\
  \texttt{\href{mailto:carrillo@imperial.ac.uk}{carrillo@imperial.ac.uk}}
  \and
  \textsc{Manuel Pájaro}
  \\
  IIM-CSIC
  \\Spanish Council for Scientific Research
  \\Eduardo Cabello 6
  \\36208 Vigo, Spain
  \\
  \texttt{\href{mailto:mpajaro@iim.csic.es}{mpajaro@iim.csic.es}}
}

\begin{document}

\title{Exponential equilibration of genetic circuits using entropy methods}


\maketitle
		
\begin{abstract}  
  We analyse a continuum model for genetic circuits based on a partial
  integro-differential equation initially proposed in Friedman, Cai \&
  Xie (2006) \cite{Friedman} as an approximation of a chemical master
  equation. We use entropy methods to show exponentially fast
  convergence to equilibrium for this model with explicit bounds. The
  asymptotic equilibration for the multidimensional case of more than
  one gene is also obtained under suitable assumptions on the
  equilibrium stationary states. The asymptotic equilibration property
  for networks involving one and more than one gene is investigated
  via numerical simulations.
\end{abstract}

                
\section{Introduction}

Translation of the information encoded in genes is responsible for all
cellular functions.  The decoding of DNA can be summarised, following
the central dogma of molecular biology, in two steps: the
transcription into messenger RNA and the translation into
proteins. Cells produce responses to environmental signals, thanks to
the regulation of DNA expression via certain feedback mechanism
activating or inhibiting the genes. Typically, regulation is produced
by the union of proteins to the DNA binding sites. Moreover, the
number of species involved in gene regulatory networks (gene
expression together with their regulation) is small, which makes its
behaviour inherently stochastic
\cite{elowitzetal02,Gillespie,koernetal05,McadamsArkin97,paulsson04}. This
underlying stochastic behaviour in gene regulatory networks is
captured by using the chemical master equation (CME) \cite{kepler,
  mackeyetal11, paulsson, Sherman}. However, the CME solution is
unavailable in most cases, due to the large (even infinite) number of
coupled equations.

There are two main ways to obtain the CME solution: via stochastic
simulation or via approximations of the CME. One of the most extended
methods to reproduce the CME dynamics using stochastic realisations is
the Stochastic Simulation Algorithm (SSA)
\cite{gillespie_76,Gillespie}. This method has no restrictions in its
applicability, even though it is computationally expensive. On the
other hand, CME approximations which remain valid under certain
conditions include the finite state projection \cite{munsky}, moment
methods \cite{engblom06,hasenaueretal14}, linear noise approximations
\cite{thomasetal14,vankampen,wallaceetal12} or hybrid models
\cite{jahnke11}.

In addition to the above mentioned methods, assuming that protein
production takes place in bursts one can obtain a partial
integro-differential equation (PIDE) as a continuous approximation of
the CME. This PIDE has a mathematical structure very similar to
kinetic and transport equations in mathematical biology
\cite{bookperthame} and it admits an analytical solution for its
steady state in the case of networks involving only one gene. In the
next subsections, we describe both the one dimensional PIDE model
\cite{Friedman} for self-regulated gene networks and the generalised
PIDE model \cite{pajaroetal17} for arbitrary genetic circuits. We will
discuss the main properties of the stationary states in one dimension
to finally explain the main results of this work.

\subsection{1-dimensional PIDE model}

The kinetic equation, first proposed by Friedman et al.~\cite{Friedman}, is a continuous approximation of the CME for gene self-regulatory networks. A
schematic representation of this genetic circuit is illustrated in
Figure \ref{fig:reactionscheme}, where the transcription-translation
mechanism from DNA to a protein $X$ is shown. Note that DNA
transcribes into messenger RNA not only from the active state at rate
(per unit time $\tau$) $k_m$, but also from the inactive state with rate
constant $k_{\varepsilon}$ lower than $k_m$, which is known as
\emph{basal transcription level} or \emph{transcriptional leakage}
\cite{Friedman,Ochab-MarcinekTabaka,pajaro15}. The messenger RNA
transcribes into protein $X$ following a first-order process with rate
constant (per unit time) $k_x$. The messenger RNA and protein are
degraded at rate constants $\gamma_m$ and $\gamma_x$ respectively.
\begin{figure}[H]
  \begin{displaymath}
    \xymatrix{&&&&\\
      \text{DNA\textsubscript{off}}
      \ar@{..>}@/^2pc/[rrr]^{\text{$k_{\epsilon}^{}$}}
      & \underset{k_{\mathrm{off}}^{}}{\overset{k_{\mathrm{on}}^{}}{
          \rightleftharpoons}}
      &  \text{DNA\textsubscript{on}}  \ar@{->}[r]^{\text{$k_m$}}
      & \text{mRNA} \ar@{->}[d]^{\text{$\gamma_m^{}$}} \ar@{->}[r]^{\text{$k_x^{}$}} & X \ar@{->}[d]^{\text{$\gamma_x^{}$}} \\
      & X \ar@{-->}[u] & & \text{$\emptyset$} & \text{$\emptyset$}
   }
  \end{displaymath}  
  \caption{Schematic representation of the transcription-translation
    mechanism under study. The promoters associated with the gene of
    interest are assumed to switch between active
    (DNA\textsubscript{on}) and inactive (DNA\textsubscript{off}) states,
    with rate constants $k_{\mathrm{on}}$ and $k_{\mathrm{off}}$
    per unit time, respectively. In this study, the transition is
    assumed to be controlled by a feedback mechanism induced by the
    binding/unbinding of a given number of $X$-protein molecules, what
    makes the network self-regulated. Transcription of messenger RNA
    (mRNA) from the active DNA form, and translation into protein
    $X$ are assumed to occur at rates (per unit time) $k_m$ and
    $k_x$, respectively. $k_{\varepsilon}$ is the rate constant
    associated with transcriptional leakage. The mRNA and protein
    degradations are assumed to occur by first order processes with
    rate constants ${\gamma}_m$ and ${\gamma}_x$, respectively.}
	\label{fig:reactionscheme}
\end{figure}

For self-regulated gene networks, activation or inhibition of the DNA
promoter is produced by the union of the protein expressed to the DNA
binding sites (feedback mechanism). So that, under protein action the
promoter can switch between its inactive (DNA\textsubscript{off}) and
active (DNA\textsubscript{on}) forms, with rate constants
$k_{\mathrm{on}}$ and $k_{\text{off}}$ respectively (see Figure
\ref{fig:reactionscheme}). There are two types of feedback mechanism:
positive or negative, corresponding to whether the protein inhibits
or promotes their production, respectively. The fraction of the
promoter in the active or inactive state is typically described by
Hill functions \cite{alon}. We can express the probability that the
promoter is in its inactive state as a function of the protein amount
$x$, denoted by $\rho:{\mathbb{R}}_{+} \rightarrow [0, \ 1]$ (see
\cite{Ochab-MarcinekTabaka,pajaro15}):
\begin{equation}\label{eq:rho}
\rho(x)=\dfrac{x^H}{x^H+K^H},
\end{equation}
where $K := \frac{k_{\text{off}}}{k_{\text{on}}}$ is the equilibrium
binding constant and $H\in \mathbb{Z}\backslash\{0\}$ is the Hill
coefficient which is positive if $H$ proteins bound to the DNA
inhibiting their production (negative feedback) and negative if $|H|$
proteins bound to the DNA activating their production (positive
feedback). Then, the rate $R_T$ of messenger RNA production
(transcription) can be written as function of the Hill expression
(\ref{eq:rho}), $R_T=k_m c(x)$, with the input function
$c(x):= \left(1 - \rho(x)\right) + \rho(x)\varepsilon$, where
$\varepsilon$ is the leakage constant defined as
$\varepsilon:=\frac{k_{\varepsilon}}{k_m}$. Note that the function
$R_T$ accounts for the messenger RNA production both from the DNA
active state (with probability $1 - \rho(x)$) with rate constant $k_m$
and from the inactive DNA (with probability $\rho(x)$) with lower rate
constant $k_{\varepsilon}$.

The PIDE model is valid under the assumption of protein production in
bursts. So, we consider gene self-regulatory networks where the
degradation rate of $mRNA$ is much faster than the corresponding to
protein, ${\gamma}_m / {\gamma}_x \gg 1$. Such condition is
verified in many gene regulatory networks, both in prokaryotic and
eukaryotic organisms \cite{ShahrezaeiSwain08,Daretal12}, and results
in protein being produced in bursts. As suggested in
\cite{Friedman,Elgartetal}, the burst size (denoted by
$b=\frac{k_x}{\gamma_m}$) is typically modelled by an exponential
distribution. The conditional probability for protein level to jump
from a state $y$ to a state $x > y$ after a burst is proportional to:
\begin{equation}
  \label{eq:bursts}
  \omega(x-y)=\dfrac{1}{b}\exp\left(-\dfrac{x-y}{b}\right),
  \qquad \text{for $x > y > 0$}.
\end{equation}
The temporal evolution of the probability density function of the
amount of proteins,
$p:{\mathbb{R}}_{+}\times{\mathbb{R}}_{+} \rightarrow
{\mathbb{R}}_{+}$ is described by the following PIDE model:
\begin{equation}\label{eq:Fried_good}
  \dfrac{\partial p}{\partial t}(t, x)
  - \dfrac{\partial (xp)}{\partial x}(t, x)
  = a \int_0^x \! \omega(x-y)c(y)p(t,y) \, \mathrm{d}y - ac(x)p(t,x), 
\end{equation}
where $\tau$ is time, $t = {\gamma}_x \tau$ represents a dimensionless
time associated to the time scale of protein degradation,
$a=\frac{k_m}{\gamma_x}$ is the dimensionless rate constant related to
transcription, which represents the mean number of bursts (burst
frequency) and $\omega(x-y)$ is given by (\ref{eq:bursts}). The input
function $c:{\mathbb{R}}_{+} \rightarrow [\varepsilon, \ 1] $, which
represents the feedback mechanism, takes the form
\cite{Ochab-MarcinekTabaka,pajaro15}:
\begin{equation}
  \label{eq:c-def}
  c(x)=\frac{K^H+ \varepsilon x^H}{K^H + x^H},
  \qquad x > 0.
\end{equation}
Note that the above input function can be constant, equal to one, when
the protein does not promote or repress its production (open
loop). This constant $c(x)=1$ is used when the DNA is always in its
active state, thus implying a unique messenger RNA production rate
($k_m$), reducing the system complexity.

We denote the \emph{stationary solution} of equation
(\ref{eq:Fried_good}) (which we sometimes call \emph{equilibrium}) as
$P_{\infty}(x)$, which therefore verifies the following equation:
\begin{equation}\label{eq:Fried_infy}
\dfrac{\partial [xP_{\infty}(x)]}{\partial x} = -a \int_0^x  \omega(x-y)c(y)P_{\infty}(y) \mathrm{d}y + ac(x)P_{\infty}(x).
\end{equation}
We say a stationary solution is normalised when its integral over
$[0,+\infty)$ (which we sometimes call its \emph{mass}) is equal to
$1$. This equation has a unique solution with mass $1$, which can be
written out explicitly as \cite{Ochab-MarcinekTabaka,pajaro15}:
\begin{equation}
  \label{eq:analytic_sol}
  P_{\infty}(x):=
  Z \left[ \rho(x) \right]^{\frac{a(1-\varepsilon)}{H}}
  x^{-(1-a\varepsilon)}e^{\frac{-x}{b}}
  =
  Z\left[ x^H +K^H \right]^{\frac{a(\varepsilon-1)}{H}}
  x^{a-1}e^{\frac{-x}{b}},
\end{equation}
with $\rho(x)$ defined in (\ref{eq:rho}) and $Z$ being a normalising
constant such that $\int_{0}^{\infty}P_{\infty}(x)\, \mathrm{d}x=1$.
In case of no self-regulation (open loop network with $c(x)=1$; that
is, $\epsilon = 1$) the stationary solution is a gamma distribution
\cite{Friedman}, which is in fact the limit of \eqref{eq:analytic_sol}
as $\epsilon$ tends to $1$:
\begin{equation}\label{invlapfri}
P_{\infty}(x):=\dfrac{x^{a-1}e^{-x/b}}{b^a \Gamma(a)},
\end{equation}
which is a limiting case of \eqref{eq:analytic_sol} when $\epsilon \to
1$.


\subsection{Generalised $n$-dimensional PIDE model}

Recently the 1D PIDE model has been extended to overcome more general
gene regulatory networks than the self-regulation considered by
Friedman \cite{Friedman}. As a first step in this extension, Bokes et
al. \cite{bokesetal2015} propose the use of variable protein
degradation rate, in order to accommodate gene networks with decoy
binding sites \cite{leeetal2012} to the PIDE model structure. Finally,
including the previous models and considering genetic networks
involving more than one gene Pájaro et al. \cite{pajaroetal17} proposed
the generalised PIDE model for any number of genes.

In \cite{pajaroetal17} a general gene regulatory network comprising
$n$ genes, $\boldsymbol{G}=\{DNA_1, \cdots, DNA_i, \cdots, DNA_n\}$,
is proposed. These genes encoded by DNA-subchains are transcribed into
$n$ different messenger RNAs $\boldsymbol{M}=\{mRNA_1,$
$ \cdots, mRNA_i, \cdots, mRNA_n\}$, which are translated into $n$
proteins types $\boldsymbol{X}=\{X_1, \cdots,$ $ X_i, \cdots,
X_n\}$.
We show a schematic representation of the general network in Figure
\ref{fig:reactionscheme_nD}, which is similar to the self-regulation
circuit. The main differences are that: (i) each DNA type can be
regulated by others different proteins than the one expressed by the
considered gene (cross regulation), and (ii) the protein degradation
rate can be a variable function of all proteins types considered.

The structure of this multidimensional network is equivalent to the previous self-regulation case. Each promoter can switch from the inactive states ($DNAi_{\mathrm{off}}$) to the active one ($DNAi_{\mathrm{on}}$) or vice versa with rate constants
$k_{\mathrm{on}}^{i}$ and $k_{\mathrm{off}}^{i}$ respectively. The
leakage (basal) messenger RNA production from the inactive promoter is
conserved at lower rate constant ($k_{\varepsilon}^i$) than its
production from the active state ($k_m^i$). Each $i$ messenger RNA
type is translated into the protein $X_i$ at rate constant
$k_x^i$. Both messengers RNA and proteins are degraded with rates
$\gamma_m^i$ and $\gamma_x^i(\mathbf{x})$ respectively.

Note that for this general network the total rate of production of $mRNA_i$, $R_T^i$, can be written as the rate constant production from the active $DNA_i$ state times one input function $c_i(\mathbf{x})$ describing all possible types of feedback mechanism. However, there are not universal expressions for $c_i(\mathbf{x})$, due to their dependence on the regulatory mechanism considered (the messenger RNA production can occur from intermediate DNA states between the total activated and the total repressed ones), some examples have been described in \cite{alon,pajaroetal17}. Without lost of generality, we can construct the input function verifying that its image is a positive interval, $c_i:{\mathbb{R}}_{+}^{n} \rightarrow [\varepsilon_i, \ 1]$, where the leakage constant $\varepsilon_i$ is defined as $k_{\varepsilon}^i/k_m^i$ with $k_{\varepsilon}^i$ being the $mRNAi$ rate constant from the total repressed $DNA_i$ (the lowest rate of $mRNA_i$ production).

\begin{figure}[H]
  \begin{displaymath}
    \xymatrix{&&&&\\
      DNAi_{\mathrm{off}}  \ar@{..>}@/^2pc/[rrr]^{\text{$k_{\epsilon}^{i}$}} & \underset{k_{\mathrm{off}}^{i}}{\overset{k_{\mathrm{on}}^{i}}{ \rightleftharpoons}} &  DNAi_{\mathrm{on}}  \ar@{->}[r]^{\text{$k_m^i$}} & mRNA_i \ar@{->}[d]^{\text{$\gamma_m^{i}$}} \ar@{->}[r]^{\text{$k_x^{i}$}} & X_i \ar@{->}[d]^{\text{$\gamma_x^{i}(\mathbf{x})$}} \\
      & X_J \ar@{-->}[u] & & \text{$\emptyset$} & \text{$\emptyset$}
    }
  \end{displaymath}
  \caption{Schematic representation of the transcription-translation
    mechanism under study. The promoters associated with the genes of
    interest are assumed to switch between active
    ($DNAi_{\mathrm{on}}$) and inactive ($DNAi_{\mathrm{off}}$)
    states, with rate constants $k_{\mathrm{on}}^{i}$ and
    $k_{\mathrm{off}}^{i}$ per unit time, respectively. The transition
    is assumed to be controlled by a feedback mechanism induced by the
    binding/unbinding of a given number of $X_j$-protein molecules
    with $j \in J$ (more than one protein type can bind to the DNA),
    which makes the network self-regulated if $i=j$ or cross-regulated
    if $j\ne i$. Transcription of messenger RNA ($mRNA_i$) from the
    active $DNAi$ form, and translation into protein $X_i$ are assumed
    to occur at rates (per unit time) $k_m^{i}$ and $k_x^{i}$,
    respectively. $k_{\varepsilon}^{i}$ is the rate constant
    associated with transcriptional leakage. The $mRNA_i$ degradation
    is assumed to occur by first order processes with rate constant
    ${\gamma}_m^{i}$. Degradation of the $X_i$-protein may follow
    different pathways, which is modelled by the function
    ${\gamma}_x^{i} (\mathbf{x})$, with
    $\gamma_x^{i}:{\mathbb{R}}_{+}^{n} \rightarrow
    {\mathbb{R}}_{+}$.} \label{fig:reactionscheme_nD}
\end{figure}

Considering the set of $n$ proteins $\mathbf{X}=\{X_1, \cdots, X_n\}$, we define the $n$-vector $\mathbf{x}=(x_1, \cdots, x_n) \in \mathbb{R}_+^n$ as the amount of each protein type. The generalised ($n$-dimensional) PIDE model, proposed in \cite{pajaroetal17}, describes the temporal evolution of the joint density distribution function of $n$ proteins $p:{\mathbb{R}}_{+}\times{\mathbb{R}}_{+}^{n} \rightarrow {\mathbb{R}}_{+}$:
\begin{equation}\label{eq:MEF_generalise}
\dfrac{\partial p}{\partial t}(t,\mathbf{x}) =  \sum_{i=1}^{n}\left( \dfrac{\partial}{\partial x_i}\left[\gamma_x^i(\mathbf{x}) x_i p(\mathbf{x})\right] + k_m^i \int_0^{x_i} \! \omega_i(x_i-y_i) c_i(\mathbf{y}_i)p(t,\mathbf{y}_i) \, \mathrm{d}y_i -k_m^ic_i(\mathbf{x})p(\mathbf{x})\right)
\end{equation}
where $\mathbf{y}_i$ represents the vector state $\mathbf{x}$ with its
$i$-th position changed to $y_i$, (that is:
$(\mathbf{y}_i)_j=x_j \ \text{if} \ j\ne i $ and
$(\mathbf{y}_i)_j=y_i \ \text{if} \ j=i $), and
$\gamma_x^i(\mathbf{x})$ is the degradation rate function of each
protein. The first term in the right-hand side of the equation
accounts for protein degradation whereas the integral describes
protein production by bursts. The burst size is assumed to follow an
exponential distribution, what leads to the conditional probability
for protein jumping from a state $y_i$ to a state $x_i$ after a burst
be given by:
\begin{equation*}
  \omega_i(x_i - y_i) = \dfrac{1}{b_i}
  \exp \left( -\dfrac{x_i-y_i}{b_i} \right)
\end{equation*}
where $b_i=\frac{k_x^{i}}{\gamma_m^{i}}$ are dimensionless frequencies
associated to translation which corresponds with the mean protein
produced per burst (burst size). The function $c_i(\mathbf{x})$
($c_i:{\mathbb{R}}_{+}^{n} \rightarrow [\varepsilon_i, \ 1]$) is an
input function, which models the regulation mechanism of the network
considered.

The stationary solution $P_{\infty}(\mathbf{x})$ of
(\ref{eq:MEF_generalise}) satisfies:
\begin{equation}\label{eq:nFried_infy}
\sum_{i=1}^{n}\left( \dfrac{\partial}{\partial x_i}\left[\gamma_x^i(\mathbf{x}) x_i P_{\infty}(\mathbf{x})\right] + k_m^i \int_0^{x_i} \! \omega_i(x_i-y_i) c_i(\mathbf{y}_i)P_{\infty}(\mathbf{y}_i) \, \mathrm{d}y_i -k_m^ic_i(\mathbf{x})P_{\infty}(\mathbf{x})\right)=0. 
\end{equation}
Note that an analytical expression for the steady state solution is
not known for the general case of the PIDE model
(\ref{eq:MEF_generalise}). Some properties of the 1D solution remain
valid for the nD steady state since $P_{\infty}(\mathbf{x})$ is a
probability density function, then
$\int_{\mathbb{R}_+^n}P_{\infty}(\mathbf{x})\, \mathrm{d} \mathbf{x} =
1$.
However, we do not have any other prior information about the
properties of stationary solutions.


\subsection{Main results}
\label{sec:main_results}

In this work we will apply entropy methods in order to analyse the
asymptotic equilibration for the kinetic equations
\eqref{eq:Fried_good} and \eqref{eq:MEF_generalise}. These equations
bear a similar structure to the self-similar fragmentation and the
growth-fragmentation equations \cite{PR,LP,DG,Caceresetal11,BCG}, used
for instance in cell division modelling. In those cases, the transport
term makes the cluster size of particles grow while the integral
term breaks the particles into pieces of smaller size. In our present
models, the transport term degrades the number density of proteins
while the integral term makes the protein number density to grow.

In fact, the kinetic equations \eqref{eq:Fried_good} and
\eqref{eq:MEF_generalise} have the structure of linear population
models as in \cite{Micheletal04,Micheletal05,Carrilloetal11} for which
the so-called general relative entropy applies. This fact already
reported in \cite{pajaroetal16b} implies the existence of infinitely
many Lyapunov functionals for these models useful for different
purposes among which to analyse their asymptotic behavior. We will
make a summary of the main properties of equation
\eqref{eq:Fried_good} in Section 2 together with a quick treatment of
the well-posedness theory for these models. They are easily
generalisable to the multidimensional case \eqref{eq:MEF_generalise}.

In sections 3 and 4, we will improve over the direct application of
the general relative entropy method in \cite{pajaroetal16b}. On one
hand, we study in Section 3 the case of gene circuits involving one
gene, equation \eqref{eq:Fried_good}, a direct functional inequality
between the $L^2$-relative entropy and its production leading to
exponential convergence. In order to fix our setting, we recall that
$\omega$ is given by \eqref{eq:bursts} for some $b > 0$, and
$c = c(x)$ is given by \eqref{eq:c-def}, for some constants $K > 0$,
$H \in \mathbb{Z} \setminus \{0\}$ and $0 < \epsilon \leq 1$; and
$a > 0$ is a constant.

\begin{theorem}[Long-time behaviour for the 1-dimensional model]
  \label{thm:exp}
  Let $p_0$ be a probability distribution such that $p_0 \in L^1((0,+\infty)) \cap L^2((0,+\infty), P_\infty^{-1})$, and let $p$ be the mild solution to
  equation \eqref{eq:Fried_good} with initial data $p_0$ (see
  Definition \ref{def:mild}). There exists a constant $\lambda > 0$
  depending only on the parameters of the equation (and not on $p_0$)
  such that
  \begin{equation*}
    \| p(t,\cdot) - P_\infty \|_{L^2((0,+\infty), P_\infty^{-1})} \leq
    e^{-\lambda t} \| p_0 - P_\infty \|_{L^2((0,+\infty), P_\infty^{-1})}\,.
  \end{equation*}
\end{theorem}
The value of $\lambda$ can be estimated explicitly from the arguments
in the proof, though we do not consider the specific value to be a
good approximation of the optimal decay rate. The behaviour of the
stationary solutions $P_\infty(x)$ near the origin and infinity is
crucial for direct functional inequalities involving the relative
entropy and its production in the one dimensional case.

Section 4 is devoted to the analysis of the multidimensional equation
\eqref{eq:MEF_generalise} corresponding to multiple genes involved in
the gene transcription. In this case, solutions to the stationary
problem \eqref{eq:nFried_infy} are not explicit and hence we are not
able to control precisely the behaviour of the stationary solutions
near the origin and infinity as before. For this reason, we are only
able to show convergence towards a unique equilibrium solution
assuming its existence with suitable behavior near the origin and
infinity:

\begin{theorem}[Long-time behaviour for the $nD$ model]
  Given any mild solution $p$ with normalised nonnegative initial data
  $p_0 \in L^1(\R_+)$ to equation \eqref{eq:MEF_generalise} and given
  a normalised stationary solution $P_\infty(\mathbf{x})$ to
  \eqref{eq:MEF_generalise} satisfying the technical assumption
  \ref{ass:ge} from Section \ref{sec:nd}, it holds that
  \begin{equation*}
    \lim\limits_{t \rightarrow \infty}
    \int_{{\mathbb{R}}_{+}^{n}}|p(t,\mathbf{x})-P_\infty(\mathbf{x})|^2
    \mathrm{d} \mathbf{x} = 0.
  \end{equation*}
  As a consequence, if a normalised stationary solution
  $P_\infty(\mathbf{x})$ of \eqref{eq:MEF_generalise} and satisfying
  assumption \ref{ass:ge} exists, it is unique.
\end{theorem}

The proof is based on a weaker variant of our one-dimensional
inequality, in which the control between the relative entropy and its
production is obtained except for an error term which happens to be
small under the assumptions of the behavior of the stationary solution
$P_{\infty}(\mathbf{x})$. Both results of equilibration are
illustrated with numerical simulations in their corresponding
sections.


\section{Mathematical preliminaries and entropy methods}

\subsection{Properties of stationary solutions}
\label{sec:equilibrium_prop}

Let us start by discussing the basic properties of the one dimensional stationary states to \eqref{eq:Fried_good}.
The behaviour of the stationary state at zero and at $+\infty$ depends
on both $r=a\varepsilon-1$ and $a$ due to the presence of the function
$\rho(x)$ and its dependence on $H$. It is as follows:
\begin{itemize}
\item[1.] If $H>0$, then $P_{\infty}(x) \simeq x^{a-1}$ as $x\to 0^+$
  and $P_{\infty}(x) \simeq x^r e^{-x/b}$ as $x\to+\infty$. Then the
  stationary state $P_\infty(x)$ exhibits a singularity at zero for
  $0<a<1$ and it is smooth otherwise having zero limit for $a> 1$ and
  a positive limit for $a=1$.
\item[2.] If $H< 0$, then $P_{\infty}(x) \simeq x^r$ as $x\to 0^+$ and
  $P_{\infty}(x) \simeq x^{a-1} e^{-x/b}$ as $x\to+\infty$. Then the
  stationary state $P_\infty(x)$ exhibits a singularity at zero for
  $a\varepsilon<1$ and it is smooth otherwise having zero limit for
  $a\varepsilon>1$ and a positive limit for $a\varepsilon=1$.
\end{itemize}
As a particular case, if $c(x) \equiv 1$ then $P_{\infty}(x)$ is given
by (\ref{invlapfri}) and we have $P_{\infty}(x) \simeq x^{a-1}$ as
$x\to 0^+$ and $P_{\infty}(x) \simeq x^{a-1} e^{-x/b}$ as
$x\to+\infty$. Then the stationary state $P_\infty(x)$ exhibits a
singularity at zero for $a<1$ and it is smooth otherwise having zero
limit for $a>1$ and a positive limit for $a=1$.

\begin{figure}[H]
	\centering
	\includegraphics[width=0.8\textwidth]{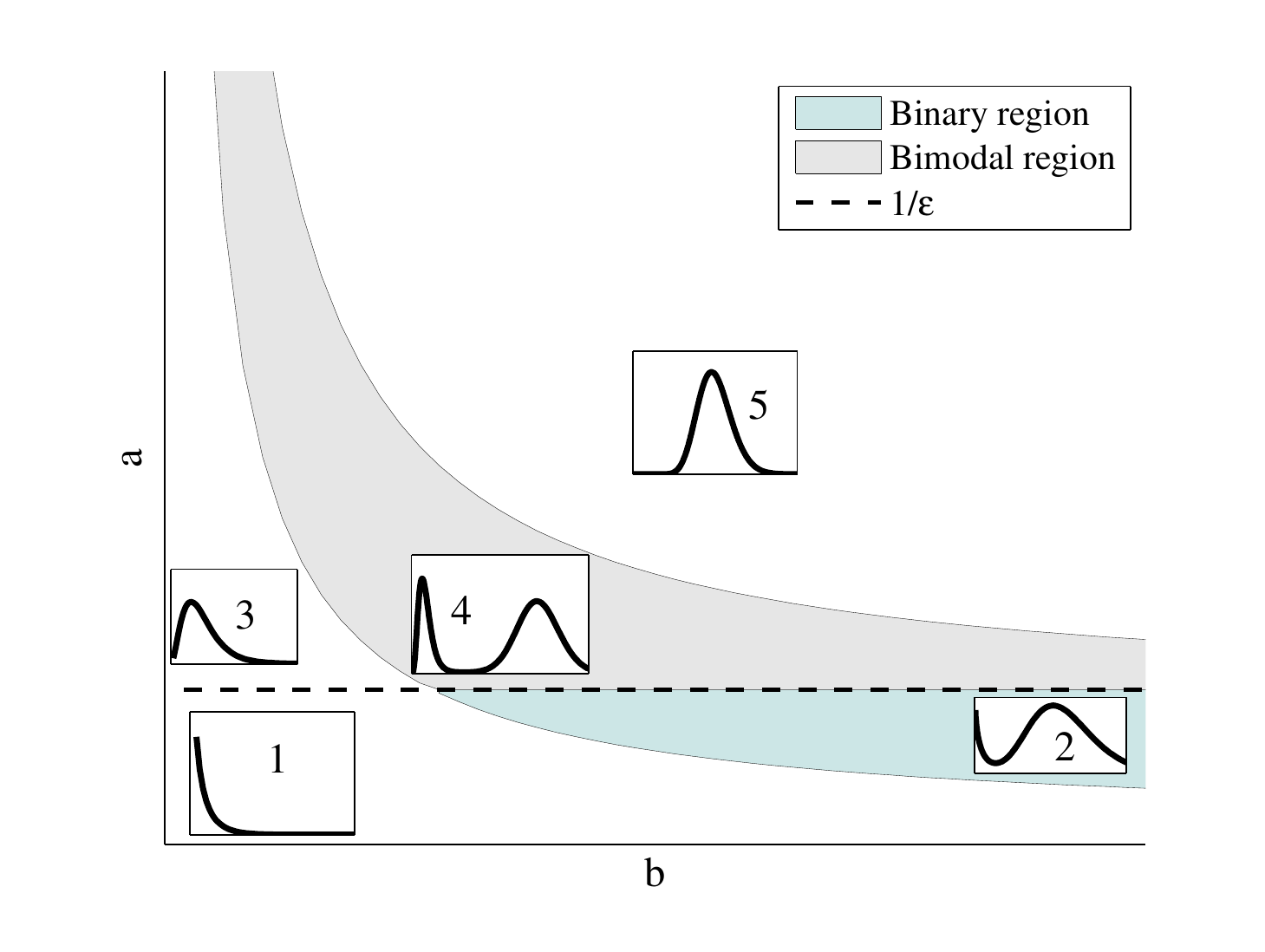}
	\caption{Regions in the parameter space, where protein distribution exhibits different behaviours for $H<0$. There are two large areas where the protein distribution change fundamentally its properties, the first including the shapes one and two, where $a<\frac{1}{\varepsilon}$ and $\lim\limits_{x \rightarrow 0}P_{\infty}(x)=+\infty$ and the second with $P_{\infty}(x)$ finite for all non-negative $x$, which includes shapes three to five. }
	\label{fig:distribution_shapes}
\end{figure}

Note that in all cases
$\lim\limits_{x \rightarrow \infty} P_{\infty}(x) = 0$. As we can see
in Fig \ref{fig:distribution_shapes}, the stationary solution has five
different qualitative behaviours for $H<0$ (see also \cite{pajaro15}):
\begin{itemize}
\item[1.] If $a<\dfrac{1}{\varepsilon}$, then
  $\lim\limits_{x \rightarrow 0} P_{\infty}(x) = \infty$.  \subitem1.1
  Only one peak in $x=0$ (Case 1 Fig
  \ref{fig:distribution_shapes}). 
  \subitem1.2 Two peaks one in $x=0$ and another in $x>0$ (Case 2 Fig
  \ref{fig:distribution_shapes}). 
\item[2.] If $a > \dfrac{1}{\varepsilon}$, then
  $\lim\limits_{x \rightarrow 0} P_{\infty}(x) = 0$. If
  $a \ge \dfrac{1}{\varepsilon}$, then
  $\lim\limits_{x \rightarrow 0} P_{\infty}(x) = M$ with $M\ge 0$
  \subitem2.1 Only one peak in $x > 0$ but close to $x=0$ (Case 3 Fig
  \ref{fig:distribution_shapes}). 
  \subitem2.2 Two different peaks at two points $x_1, x_2 > 0$ (Case 4
  Fig \ref{fig:distribution_shapes}). 
  \subitem2.3 Only one peak in $x \ge 0$ (Case 5 Fig
  \ref{fig:distribution_shapes}). 
\end{itemize}
Note that, case 2.1 and 2.3 are equivalent, and
$\lim\limits_{x \rightarrow \infty} P_{\infty}(x) = 0$ for all
cases. If $H>0$ (or $c(x)=1$) the bimodal behaviour disappears, and
only cases 3 or 5 remain for $a>1$ and case 1 if $a<1$.

\subsection{Well-posedness}

The 1D equation \eqref{eq:Fried_good} is a linear integro-differential
equation for which well-posedness and some basic properties follow
from standard methods. A \emph{classical solution} to equation
\eqref{eq:Fried_good} with initial data
$p_0 \in \mathcal{C}^1([0,+\infty))$ is a function
$p \in \mathcal{C}^1([0,+\infty) \times (0,+\infty))$ which satisfies
\eqref{eq:Fried_good} for all
$(t,x) \in [0,+\infty) \times (0,+\infty)$, and such that
$p(0,x) = p_0(x)$ for all $x \in (0,+\infty)$. It is not hard to show
that, given an integrable initial condition
$p_0 \in \mathcal{C}^{1,\mathrm{b}}([0,+\infty))$, there exists a
unique mass-conserving classical solution. In order to give a brief
sketch of the proof it is perhaps easier to work with \emph{mild
  solutions}, which we will introduce now. Given
$p(t)=p(t,\cdot) \in L^1(0,+\infty)$, we denote by $L[p(t)]$ the
right-hand side of \eqref{eq:Fried_good} given by
\begin{equation*}
  L[p(t)](x) := a \int_0^x \! \omega(x-y)c(y)p(t,y) \, \mathrm{d}y -
  ac(x)p(t,x),
  \qquad x > 0,
\end{equation*}
and given any function $p_0 \: [0,+\infty) \times (0,+\infty) \to \R$
we define
\begin{equation*}
  (X_t \# p_0) (x) := p_0(x e^t) e^t,
  \qquad \text{for $t \geq 0, x > 0$.}
\end{equation*}
This notation is motivated by the fact that $X_t \# p_0$ is the
transport of the function $p_0$ by the dilation map $X_t(x) := x e^{-t}$. By
the method of characteristics one easily sees that a classical
solution $p$ to \eqref{eq:Fried_good} must satisfy
\begin{equation}
  \label{eq:mild}
  p(t,x)
  = (X_t \# p_0)(x)
  + \int_0^t \big (X_{t-s} \# L[p(s,\cdot)] \big)(x) \d s
  \qquad
  \text{for all $t \geq 0$, $x > 0$.}
\end{equation}
This suggests the following definition.
\begin{definition}
  \label{def:mild}
  Let $p_0 \in L^1(0,+\infty)$. We say that
  $p \in \mathcal{C}([0,\infty); L^1(0,+\infty))$ is a \emph{mild
    solution} to equation \eqref{eq:Fried_good} with initial data
  $p_0$ if it satisfies \eqref{eq:mild} for all $t \geq 0$, for almost
  all $x > 0$.
\end{definition}

\begin{theorem}
  \label{thm:exist_mild}
  For any $p_0 \in L^1(0,+\infty)$ there exists a unique mild solution
  of \eqref{eq:Fried_good} with initial data $p_0$ satisfying
  $$
  \int_0^\infty p(t,x)\, \mathrm{d}x = \int_0^\infty p_0(x)\, \mathrm{d}x 
  \qquad \text{for all $t \geq 0$.}
  $$ 
  In addition, there is a constant $C > 0$ (independent of $p_0$) such that
  \begin{equation}
    \label{eq:stability-L1}
    \|p(t) \|_1 \leq e^{Ct} \|p_0\|_1
    \qquad \text{for all $t \geq 0$.}
  \end{equation}
  Moreover, for any $p_0 \in \mathcal{C}^{1,\mathrm{b}}(0,+\infty)$
  there exists a unique classical solution of \eqref{eq:Fried_good}
  with initial data $p_0$.
\end{theorem}

\begin{proof}
  This result can be obtained by considering the functional:
  \begin{equation*}
    \Phi[p](t,x) := 
    (X_t \# p_0)(x)
    + \int_0^t \big (X_{t-s} \# L[p(s,\cdot)] \big)(x) \d s,
  \end{equation*}
  defined on the Banach space
  \begin{equation*}
    Y := \{ p \in \mathcal{C}([0,T]; L^1(0,+\infty))
    \mid p(0) = p_0 \}
  \end{equation*}
  with norm
  \begin{equation*}
    \| p \|_Y := \sup_{t \in [0,T]} \| p_t \|_1,
  \end{equation*}
  for $T > 0$ small enough. Note that
  $$
  \int_0^\infty \Phi[p](t,x)\, \mathrm{d}x = \int_0^\infty p(t,x)\, \mathrm{d}x = \int_0^\infty p_0(x)\, \mathrm{d}x 
    \qquad \text{for all $t \geq 0$.}
  $$  
  By following an argument very similar to that of Picard iterations,
  one obtains the existence of mild solutions on a time interval
  $[0,T]$. Since the equation is linear (and our equation is invariant
  under time translations), this argument can be iterated to find
  solutions on $[0,+\infty)$. We refer to \cite{EngNag,CCC} for full
  details of this standard argument.

If the initial condition $p_0$ is in $\mathcal{C}^{1, \mathrm{b}}(0,+\infty)$, one
can see that the iteration above can also be done in the
space $Z := \{ p \in \mathcal{C}^{1,\mathrm{b}}([0,T] \times (0,+\infty))
\mid p(0,x) = p_0(x) \text{ for $x > 0$} \}$. This gives the existence of a unique classical solution in this space.
\end{proof}

The constructed solutions have basic properties: positivity
preserving, $L^1$-contraction, and maximum principle.

\begin{lemma}
  \label{lem:positivity}
  Take $p_0 \in L^1(0,+\infty)$ and let $p$ be the unique mild
  solution to equation \eqref{eq:Fried_good} given by Theorem
  \ref{thm:exist_mild}.
  \begin{enumerate}
  \item Positivity is preserved: if $p_0 \geq 0$ a.e. then
    $p(t) \geq 0$ a.e., for all $t \geq 0$.
  \item The $L^1$ norm is decreasing
    \begin{equation*}
      \|p(t) \|_1 \leq \|p_0\|_1
      \qquad \text{for all $t \geq 0$,}
    \end{equation*}
    leading to $L^1$-contraction by linearity. If $p_0 \geq 0$, the above inequality becomes an identity.
  \item Maximum principle:
    \begin{equation*}
      \inf_{x > 0} \frac{p_0(x)}{P_\infty(x)}
      \leq
      \frac{p(t,x)}{P_\infty(x)}
      \leq
      \sup_{x > 0} \frac{p_0(x)}{P_\infty(x)}.
    \end{equation*}
  \end{enumerate}
\end{lemma}

\begin{proof}
  In order to show that positivity is preserved for any classical
  solution, we can rewrite, using Duhamel's formula,
  \begin{equation*}
    p(t) = S_t p_0 + \int_0^t S_{t-s} L^+[p(s)] \,\d s,
  \end{equation*}
  where $S_t$ is the semigroup associated to the equation
  $\partial_t p - \partial_x (x p) + a c(x) p = 0$ and $L^+$ is the
  operator given by
  \begin{equation*}
    L^+[p(t)](x) := a \int_0^x \omega(x-y)c(y)p(t,y) \,\d y
    \qquad x > 0.
  \end{equation*}
  This way of writing the solution clearly shows $p$ is nonnegative if
  $p_0$ is nonnegative. Now, for a mild solution we obtain the same
  result by approximation from classical solutions, taking into account the $L^1$-stability
  \eqref{eq:stability-L1}.

  For the second part of the result, denote by $T_t$ the semigroup in
  $L^1(0,+\infty)$ defined by the equation, and write
  $f_+ := \max\{0, f\}$, $f_- := \max\{0, -f\}$ for the positive and
  negative parts of a function $f$, so that $f = f_+ - f_-$. The
  positivity and mass preservation imply that:
  \begin{multline*}
    \| p(t) \|_1 = \| T_t p_0 \|_1
    \leq \|T_t ((p_0)_+) \|_1 + \|T_t ((p_0)_-) \|_1
    \\
    = \int T_t ((p_0)_+) + \int T_t ((p_0)_-)
    = \int (p_0)_+ + \int (p_0)_-
    = \| p_0 \|_1.
  \end{multline*}

  Finally, for the maximum principle just notice that, if $M$ is the
  supremum on the right hand side, the function $q = M P_\infty - p$
  is a mild solution with nonnegative initial data. Due to
  preservation of positivity we obtain the inequality on the
  right-hand side. The minimum principle is obtained analogously.
\end{proof}

\subsection{Entropy and $H$-theorem}
\label{sec:entropy}

Let $H \: [0,+\infty) \to \R$ be a convex function. We define the general
relative entropy functional as:
\begin{equation}\label{eq:def_GRE}
  \mathcal{G}_H(u)(t)= \int_{0}^{\infty} H(u(t,x)) P_{\infty}(x) \mathrm{d}x,
\end{equation}
with $u(t,x):=p(t,x)/P_\infty(x)$. The basic general relative entropy
principle is that $\mathcal{G}_H(p(t)/P_\infty)$ is a decreasing
quantity when $p(t)$ is a solution to \eqref{eq:Fried_good}, see
\cite{Micheletal04,Micheletal05,Carrilloetal11,pajaroetal16b}.

\begin{proposition}
  \label{prop:Hconvex_entropy1D}
  Let $H \: [0,+\infty) \to \R$ is a convex function in
  $\mathcal{C}^1([0,+\infty))$ and let $p$ be a classical solution to
  \eqref{eq:Fried_good} with integrable initial condition
  $p_0 \in \mathcal{C}^{1,b}[0,+\infty)$ such that
  $|p_0(x)| \leq M P_\infty(x)$ for some $M > 0$. Thus, the relative entropy satisfies
  \begin{multline}
    \label{eq:Dt_Gentropy_eqn}
    \dfrac{\mathrm{d}\mathcal{G}_H(u)}{\mathrm{d} t}
    = a\!\int_{0}^{\infty}\!\!\!\!\int_{y}^{\infty}\!\!\omega(x-y)\Big( H(u(x)) - H(u(y)) +H'(u(x))\left(u(y)-u(x)\right) \Big)c(y)P_{\infty}(y) \,\mathrm{d} x \mathrm{d} y\leq 0\,,
  \end{multline}
for all $t\geq 0$.
\end{proposition}

\begin{remark}
  Notice that the dependence on the time variable in
  \eqref{eq:Dt_Gentropy_eqn} has been omitted for simplicity. Observe
  that the right-hand side in \eqref{eq:Dt_Gentropy_eqn} is
  non-positive since the convexity of $H$ implies
  $H(u)-H(v)+H'(u)(v-u) \le 0$ for all $u,v\in\R$.
\end{remark}

Proposition \ref{prop:Hconvex_entropy1D} is very close to the results
in Section 2 of \cite{Micheletal05}, but is strictly not contained
there due to the form of the integral operator. It is worth giving a
derivation of the result, so we include a proof here. We first obtain
a technical lemma involving some classical computations in
\cite{Micheletal05}:

\begin{lemma}
  \label{lemma:l2}
  Under the assumptions of Proposition
  \ref{prop:Hconvex_entropy1D}, then the following equality is satisfied
  \begin{equation}\label{eq:auxl1}
    H'(u(x))\dfrac{\partial [xp(x)]}{\partial x}
    =
    \dfrac{\partial [H(u(x))xP_{\infty}(x)]}{\partial x}
    + \left[u(x)H'(u(x))-H(u(x))\right]
    \dfrac{\partial [xP_{\infty}(x)]}{\partial x}\,.
  \end{equation}
\end{lemma}

\begin{proof}
  We know that
  \begin{equation*}
    \dfrac{\partial H(u(x))}{\partial x}=H'(u(x))\dfrac{\partial u}{\partial x}=\dfrac{H'(u(x))}{P_{\infty}(x)}\left(\dfrac{\partial p}{\partial x} -u(x)\dfrac{\partial P_{\infty}}{\partial x} \right),
  \end{equation*}
  and
  \begin{equation*}
    \dfrac{\partial [H(u(x))xP_{\infty}(x)]}{\partial x}= xP_{\infty}(x)\dfrac{\partial H(u(x))}{\partial x} +H(u(x))\dfrac{\partial [xP_{\infty}(x)]}{\partial x}.
  \end{equation*}
  So that, replacing the first expression in the second we have that:
  \begin{equation}\label{eq:aux1_lemma1}
    \dfrac{\partial [H(u(x))xP_{\infty}(x)]}{\partial x}= xH'(u(x))\left(\dfrac{\partial p}{\partial x} -u(x)\dfrac{\partial P_{\infty}}{\partial x} \right) +H(u(x))\dfrac{\partial [xP_{\infty}(x)]}{\partial x}.
  \end{equation}
  Next, by using the following identities:
  \begin{equation*}
    x\dfrac{\partial p}{\partial x}= \dfrac{\partial[x p(x)]}{\partial x}-p(x) ~~\text{and}~~ x\dfrac{\partial P_{\infty}}{\partial x}= \dfrac{\partial[x P_{\infty}(x)]}{\partial x}-P_{\infty}(x),
  \end{equation*}
  in (\ref{eq:aux1_lemma1}) we obtain:
  \begin{align*}
    \dfrac{\partial [H(u(x))xP_{\infty}(x)]}{\partial x} =&\, H'(u(x))\left(\dfrac{\partial[x p(x)]}{\partial x} -p(x) -u(x)\left(\dfrac{\partial [xP_{\infty}(x)]}{\partial x} - P_{\infty}(x)\right)\right)\\
                                                          & +H(u(x))\dfrac{\partial [xP_{\infty}(x)]}{\partial x}\\
    =&\, H'(u(x))\left(\dfrac{\partial[x p(x)]}{\partial x} -u(x)\dfrac{\partial [xP_{\infty}(x)]}{\partial x} \right) +H(u(x))\dfrac{\partial [xP_{\infty}(x)]}{\partial x}.
  \end{align*}
  Note that the terms $u(x)P_{\infty}(x)-p(x)$ vanish, since
  $u(x)P_{\infty}=p(x)$. Finally, reordering terms in the last
  equation we obtain the equality (\ref{eq:auxl1}).
\end{proof}

\begin{proof}[Proof of Proposition
  \ref{prop:Hconvex_entropy1D}]
  We start the proof computing the time derivative of the general
  relative entropy functional
  \begin{align*}
    \dfrac{\mathrm{d}\mathcal{G}_H(u)}{\mathrm{d}t}=  \dfrac{\partial }{\partial t} \int_{0}^{\infty}H(u(x))P_{\infty}(x)\mathrm{d}x &=\int_{0}^{\infty}\dfrac{\partial }{\partial t} H(u(x))P_{\infty}(x)\mathrm{d}x=\int_{0}^{\infty}H'(u(x))\dfrac{\partial p}{\partial t} \mathrm{d}x .
  \end{align*}
  We replace the time derivative of $p(\tau,x)$ by its expression
  (\ref{eq:Fried_good}) to obtain:
  \begin{equation*}
    \dfrac{\mathrm{d}\mathcal{G}_H(u)}{\mathrm{d}t}= \int_{0}^{\infty}H'(u(x))\left( \dfrac{\partial [xp(x)]}{\partial x} + a \int_0^x  \omega(x-y)c(y)p(y) \mathrm{d}y - ac(x)p(x) \right)\mathrm{d}x .
  \end{equation*}
  Using lemma \ref{lemma:l2} and the fact that
  $p(x)=u(x)P_{\infty}(x)$ we have:
  \begin{align*}
    \dfrac{\mathrm{d}\mathcal{G}_H(u)}{\mathrm{d}t}=&\, \int_{0}^{\infty}\left(\dfrac{\partial [H(u(x))xP_{\infty}(x)]}{\partial x} + \left(u(x)H'(u(x))-H(u(x))\right)\dfrac{\partial [xP_{\infty}(x)]}{\partial x}\right)\mathrm{d}x \\	
                                                       &\, +a\int_{0}^{\infty}H'(u(x))\left( \int_0^x  \omega(x-y)c(y)u(y)P_{\infty}(y) \mathrm{d}y - c(x)u(x)P_{\infty}(x)\right) \mathrm{d}x .
  \end{align*}
  In the above equation the term
  $$
  \int_{0}^{\infty}\frac{\partial [H(u(x))xP_{\infty}(x)]}{\partial
    x}\mathrm{d}x
  $$
 vanishes since
  $\lim_{x \to +\infty} x P_\infty(x) = \lim_{x \to 0} x P_\infty(x) =
  0$,
  and noticing that $u(x) \leq M$ for all $t\geq 0$, $x > 0$
  due to the maximum principle in Lemma
  \ref{lem:positivity}. Replacing the term containing the first order derivative by its value in
  equation (\ref{eq:Fried_infy}) we get
  \begin{align*}
    \dfrac{\mathrm{d}\mathcal{G}_H(u)}{\mathrm{d}t}=&\, -a\int_{0}^{\infty}\left(u(x)H'(u(x))-H(u(x))\right) \left( \int_0^x  \omega(x-y)c(y)P_{\infty}(y) \mathrm{d}y - c(x)P_{\infty}(x)\right)\mathrm{d}x \\
                                                       &\, +a\int_{0}^{\infty}H'(u(x))\left( \int_0^x  \omega(x-y)c(y)u(y)P_{\infty}(y) \mathrm{d}y - c(x)u(x)P_{\infty}(x)\right) \mathrm{d}x .
  \end{align*}
  Reordering terms in the above equation we have that
  \begin{align*}
    \dfrac{\mathrm{d}\mathcal{G}_H(u)}{\mathrm{d}t}=&\, a\int_{0}^{\infty}H(u(x))\left( \int_0^x  \omega(x-y)c(y)P_{\infty}(y) \mathrm{d}y - c(x)P_{\infty}(x)\right)\mathrm{d}x \\
                                                       &\, +a\int_{0}^{\infty}H'(u(x)) \left( \int_0^x  \omega(x-y)c(y)u(y)P_{\infty}(y) \mathrm{d}y -u(x)\int_0^x  \omega(x-y)c(y)P_{\infty}(y) \mathrm{d}y \right) \mathrm{d}x .
  \end{align*}
  Note that
  $$\int_{0}^{\infty}H(u(x)) c(x)P_{\infty}(x)\mathrm{d}x=\int_{0}^{\infty}H(u(y)) c(y)P_{\infty}(y)\mathrm{d}y,$$ so we can change the order of integration in the above equation to obtain
  \begin{align*}
    \dfrac{\mathrm{d}\mathcal{G}_H(u)}{\mathrm{d}t}=&\, a\int_{0}^{\infty}\left( \int_y^{\infty}  \omega(x-y)H(u(x)) \mathrm{d}x c(y)P_{\infty}(y)  - H(u(y))c(y)P_{\infty}(y)\right)\mathrm{d}y \\
                                                             &\, +a\int_{0}^{\infty} \int_y^{\infty}  \omega(x-y)\left[H'(u(x))\left(u(y)-u(x)\right)\right]c(y)P_{\infty}(y) \mathrm{d} x \mathrm{d}y .
  \end{align*}
  Since $\int_{y}^{\infty}\omega(x-y)\mathrm{d}x=1$, we multiply
  by this integral the second term in the first line on the
  right-hand side of the above equation to conclude
  \begin{align*}
    \dfrac{\mathrm{d}\mathcal{G}_H(u)}{\mathrm{d}t}=&\, a\int_{0}^{\infty} \int_y^{\infty}  \omega(x-y) \left[ H(u(x)) -  H(u(y))\right]  c(y)P_{\infty}(y)\mathrm{d}x \mathrm{d}y \\
                                                       &\, +a\int_{0}^{\infty} \int_y^{\infty}  \omega(x-y)\left[H'(u(x))\left(u(y)-u(x)\right)\right]c(y)P_{\infty}(y) \mathrm{d} x \mathrm{d}y \,,
  \end{align*}
which is the desired identity.
\end{proof}
	

\section{Exponential convergence for the 1D PIDE model}
\label{sec:exp_conv}
        
In this section our aim is to prove that equation
(\ref{eq:Fried_good}) converges exponentially to the steady
state, $P_{\infty}$. For this purpose, we consider the $L^2$-relative entropy, i.e., the convex function $H$ is chosen as $H(u)=(u-1)^2$, and
\begin{align*}
\mathcal{G}_2(u)(t):=\int_{0}^{\infty}P_\infty(x)(u(t,x)-1)^2 \mathrm{d}x &= \int_{0}^{\infty}\dfrac{p^2(t,x)}{P_\infty^2(x)}P_\infty(x) \mathrm{d}x -1 = \int_{0}^{\infty} u^2(t,x) P_\infty(x) \mathrm{d}x -1,
\end{align*}
where we have used that $p(t,x)$ and $P_\infty(x)$ are probability density functions. Now, by replacing the value of the considered convex function in Proposition \ref{prop:Hconvex_entropy1D}, we obtain the following identity
\begin{equation}\label{eq:Dt_entropy_eqn}
\mathcal{D}_2(u)(t):=-\dfrac{\mathrm{d}\mathcal{G}_2(u)}{\mathrm{d} t} = a\int_{0}^{\infty}\int_{y}^{\infty}\omega(x-y)\left( u(t,x) - u(t,y) \right)^2 c(y)P_\infty(y) \mathrm{d} x \mathrm{d} y.
\end{equation}
The entropy method consists in finding conditions under which the following functional inequality holds:
\begin{equation}\label{ineq:expGRE}
\mathcal{G}_2(u) \leq \dfrac{1}{2\beta}\mathcal{D}_2(u).
\end{equation}
Notice that the dependence on the time variable can be forgotten at this point, since our objective is to show such an inequality among a subset of suitable probability densities.
For this purpose, we start by rewriting $\mathcal{G}_2(u)$ in a equivalent form \cite{Caceresetal11}:

\begin{lemma}\label{l:D2eqH2}
	Given a non-negative measurable function $P_{\infty} : (0, \ \infty)\rightarrow {\mathbb{R}}_{+}$ such that $\int_{0}^{\infty}P_{\infty}(x)\mathrm{d} x =1$ and defining the functional
	\begin{equation*}
	\mathcal{H}_2(u):=\int_{0}^{\infty}\int_{y}^{\infty}P_{\infty}(x)P_{\infty}(y)\left( u(x) - u(y) \right)^2 \mathrm{d} x \mathrm{d} y ,
	\end{equation*}
	there holds $\mathcal{G}_2(u)=\mathcal{H}_2(u)$.
\end{lemma}

\begin{proof}
Expanding the square implies
\begin{equation}\label{eq_aux_lem_GH}
\mathcal{G}_2(u)=\int_{0}^{\infty}P_{\infty}(x)(u(x)-1)^2 \mathrm{d}x= \int_{0}^{\infty}P_{\infty}(x)u(x)^2 \mathrm{d}x -1,
\end{equation}
while $\mathcal{H}_2(u)$ is a symmetric function, so that:
\begin{align*}
\mathcal{H}_2(u)(\tau) =&\, \dfrac{1}{2}\int_{0}^{\infty}\int_{0}^{\infty}P_{\infty}(x)P_{\infty}(y)\left( u(x) - u(y) \right)^2 \mathrm{d} x \mathrm{d} y  \\
=&\, \dfrac{1}{2}\int_{0}^{\infty}\int_{0}^{\infty}P_{\infty}(x)P_{\infty}(y)\left( u(x)^2 -2u(x) u(y) +u(y)^2 \right) \mathrm{d} x \mathrm{d} y \\
=&\, \int_{0}^{\infty}\int_{0}^{\infty}P_{\infty}(x)P_{\infty}(y)u(x)^2  \mathrm{d} x \mathrm{d} y - \int_{0}^{\infty}\int_{0}^{\infty}P_{\infty}(x)P_{\infty}(y)u(x) u(y)  \mathrm{d} x \mathrm{d} y \\
=&\, \int_{0}^{\infty} P_{\infty}(x)u(x)^2 \left(\int_{0}^{\infty}P_{\infty}(y)  \mathrm{d} y\right) \mathrm{d} x - \int_{0}^{\infty}\int_{0}^{\infty}p(x) p(y)  \mathrm{d} x \mathrm{d} y\\
=&\, \int_{0}^{\infty}P_{\infty}(x)u(x)^2 \mathrm{d}x -1, 
\end{align*}
which is equivalent to (\ref{eq_aux_lem_GH}).
\end{proof}

As consequence of this lemma we are reduced to show the inequality 
\begin{equation}\label{ineq:expGRE_D2}
\mathcal{H}_2(u) \leq \dfrac{1}{2\beta}\mathcal{D}_2(u),
\end{equation}
among a suitable subset of probability densities.

\subsection{Entropy-entropy production inequality}

We start by obtaining bounds for the steady state solution
$P_{\infty}$, of the Friedman equation (\ref{eq:Fried_good}).

\begin{lemma}\label{lema:Pinf_bound}
  ($P_{\infty}$ bounds) For $\delta>0$ we define the intervals of
  length $\frac{1}{2}$:
  \begin{equation*}
    I_{k,\delta}:=\left(\delta+\frac{k}{2}, \ \delta+\frac{k+1}{2} \right], \qquad k\ge 0 ~~\text{integer},
  \end{equation*}
  and 
  \begin{equation*}
    p_k:= C\left[ \left(\delta+\frac{k}{2}\right)^H +K^H \right]^{\frac{a(\varepsilon-1)}{H}}\left(\delta+\frac{k}{2}\right)^{a-1}e^{\frac{-(\delta+\frac{k}{2})}{b}}=P_{\infty}\left(\delta+\frac{k}{2}\right).
  \end{equation*}
  Then, the following inequality holds:
  \begin{equation}\label{eq:Pinf_boundary}
    A(\delta)\le \dfrac{P_{\infty}(x)}{p_k} \le B(\delta), \qquad \forall x \in I_{k,\delta} ~~\text{and}~~ \forall k,
  \end{equation}
  with $P_{\infty}(x)$ given by \eqref{eq:analytic_sol}.
\end{lemma}

\begin{proof}
  Note that $\left[ x^H +K^H \right]^{\frac{a(\varepsilon-1)}{H}}$ and
  $e^{\frac{-x}{b}}$ are decreasing functions, so that their maxima
  are at $\bar{x}_0=\delta+\frac{k}{2}$ and their minima are at
  $\bar{x}_1=\delta+\frac{k+1}{2}$ in $I_{k,\delta}$.  The term
  $x^{a-1}$ shows different behaviours which depend on the parameter
  $a$, (this term is increasing if $a>1$, constant if $a=1$ and
  decreasing if $a<1$). So that, we can bound $P_{\infty}(x)$ in the
  interval $ I_{k,\delta}$ as follows:
\begin{equation}\label{eq:Pinf_bound}
\left\lbrace
\begin{array}{ll}
g(\bar{x}_1)(\delta+\frac{k}{2})^{a-1} \le  P_{\infty}(x)  \le g(\bar{x}_0)\left(\delta+\frac{k+1}{2}\right)^{a-1}
& \qquad  \text{if}~~ a>1 \\
&\\
g(\bar{x}_1) \le  P_{\infty}(x)  \le g(\bar{x}_0)
& \qquad \text{if}~~ a=1 \\
&\\
g(\bar{x}_1)\left(\delta+\frac{k+1}{2}\right)^{a-1} \le  P_{\infty}(x)  \le g(\bar{x}_0)\left(\delta+\frac{k}{2}\right)^{a-1}
 &\qquad  \text{if}~~ a<1
\end{array}
\right.
\end{equation}
where $g(x)=Z\left[ x^H +K^H \right]^{\frac{a(\varepsilon-1)}{H}}e^{\frac{-x}{b}}$.

Now, in order to calculate the bounds of $\frac{P_{\infty}(x)}{p_k}$, we divide the expression (\ref{eq:Pinf_bound}) by $p_k$ to obtain $A(\delta,k) \le \frac{P_{\infty}(x)}{p_k} \le B(\delta,k)$ with the functions $A$ and $B$ being,
\begin{equation*}
A(\delta,k):=\left\lbrace \begin{array}{lr}
\left(\dfrac{(\delta+\frac{k+1}{2})^H +K^H}{(\delta+\frac{k}{2})^H +K^H}\right)^{\frac{a(\varepsilon-1)}{H}}e^{\frac{-1}{2b}} & \qquad  \text{if}~~ a\ge 1\\
&\\
\left(\dfrac{(\delta+\frac{k+1}{2})^H +K^H}{(\delta+\frac{k}{2})^H +K^H}\right)^{\frac{a(\varepsilon-1)}{H}}e^{\frac{-1}{2b}}\left(\dfrac{2\delta+k+1}{2\delta+k}\right)^{a-1} & \qquad  \text{if}~~ a<1
\end{array}\right.
\end{equation*}
and
\begin{equation*}
B(\delta,k):=\left\lbrace \begin{array}{lr}
\left(\dfrac{2\delta+k+1}{2\delta+k}\right)^{a-1} & \qquad \text{if}~~ a > 1\\
&\\
1 & \qquad  \text{if}~~ a \le 1
\end{array}\right.
\end{equation*}
Notice that,
\begin{equation*}
\lim\limits_{k \rightarrow \infty} A(\delta,k)=e^{-\frac{1}{2b}}, \qquad \lim\limits_{k \rightarrow \infty} B(\delta,k)=1\,,
\end{equation*}
implies that $A(\delta):=\underset{k\ge 0}{\min}\left(A(\delta,k)\right)$ and $B(\delta):=\underset{k\ge 0}{\max}\left(B(\delta,k)\right)$ are well-defined and positive, leading to desired inequality \eqref{eq:Pinf_boundary}.
\end{proof}

Note that inequality (\ref{eq:Pinf_boundary}) can be directly checked for the simplest open loop case, whose stationary solution is given by (\ref{invlapfri}).

\begin{lemma}\label{lemma:ine_aux1}
Let us define 	
\begin{equation}\label{eq:def_Mj}
	M_j:=\sum_{k=1}^{j-1}\dfrac{1}{m_k},
\end{equation} 
with $\{m_k\}_{k\ge 1}$ a positive sequence given by $m_k= p_k e^{\frac{\delta+\frac{k}{2}}{2b}}$. Then, there exists $C>0$ such that
\begin{equation}\label{xxx}
	m_k \sum_{j=k+1}^{\infty} M_j p_j \le C p_k \,,\qquad \mbox{for all } k\in\mathbb{N}\,.
\end{equation}
\end{lemma}
\begin{proof}
We define $\{a_j\}_{j\ge 1}$ with $a_j=\frac{1}{m_j}$ to calculate the following limit
\begin{equation*}
\lim\limits_{j \rightarrow \infty} \dfrac{a_{j+1}-a_j}{M_{j+1}-M_j}=\lim\limits_{j \rightarrow \infty} \left(\dfrac{\left(\left(\delta + \frac{j+1}{2}\right)^H +K^H\right)^{\frac{a(1-\varepsilon)}{H}}\left(\delta+\frac{j+1}{2}\right)^{1-a}}{\left(\left(\delta + \frac{j}{2}\right)^H +K^H\right)^{\frac{a(1-\varepsilon)}{H}}\left(\delta+\frac{j}{2}\right)^{1-a}}e^{\frac{1}{4b}} -1\right)=e^{\frac{1}{4b}} -1.
\end{equation*}
Since this limit exists and $\{M_j\}_{j\ge 1}$ is a strictly increasing and divergent sequence, we can use the Stolz-Cesàro theorem to obtain that $M_j\le C_0 a_j$, with $C_0>0$ constant.
Then,
$$m_k \sum_{j=k+1}^{\infty} M_j p_j \le C_{0}m_k \sum_{j=k+1}^{\infty} a_j p_j.$$
The summation term at the right hand side can be calculated as follows
\begin{equation*}
\sum_{j=k+1}^{\infty} a_j p_j= \sum_{j=k+1}^{\infty} e^{-\frac{2\delta+j}{4b}} =\dfrac{e^{-\frac{2b-1}{4b}}}{e-1}e^{-\frac{2\delta+k}{4b}},
\end{equation*}
so that $$m_k \sum_{j=k+1}^{\infty} M_j p_j \le C m_k e^{-\frac{2\delta+k}{4b}}= C p_k,$$ with $C=C_0 \frac{e^{-\frac{2b-1}{4b}}}{e-1}$, concluding the proof.
\end{proof}

In order to prove the exponential convergence of the Friedman equation (\ref{eq:Fried_good}) we are going to split the proof of inequality (\ref{ineq:expGRE_D2}) in the following two propositions. 

\begin{proposition}\label{prop:exp_step1}
There exists $\lambda>0$ such that
	\begin{equation}\label{eq:prop_cnpi}
	\lambda \mathcal{H}_2(u) \le \int_{0}^{\infty}\int_{y}^{y+1}P_{\infty}(y)\left( u(x) - u(y) \right)^2 \mathrm{d} x \mathrm{d} y  :=D(u),
	\end{equation}
with $u=p/P_\infty$, for all $p \in L^1((0,+\infty)) \cap L^2((0,+\infty), P_\infty^{-1})$.
\end{proposition}
\begin{proof}
We take $0<\delta<1$ and split $\mathcal{H}_2(u)$ in two parts
\begin{align*}
 \mathcal{H}_2(u)=&\int_{\delta}^{\infty}\int_{y}^{\infty}P_{\infty}(x)P_{\infty}(y)\left( u(x) - u(y) \right)^2 \mathrm{d} x \mathrm{d} y \\ &+\int_{0}^{\delta}\int_{y}^{\infty}P_{\infty}(x)P_{\infty}(y)\left( u(x) - u(y) \right)^2 \mathrm{d} x \mathrm{d} y := \mathcal{H}_{21}(u)+\mathcal{H}_{22}(u)\,.
\end{align*}
For $i,j\ge 0$ integers we define
\begin{equation*}
A_{i,j}:=\int_{I_{i,\delta}}\int_{I_{j,\delta}}\left(u(x)-u(y)\right)^2 \mathrm{d} y \mathrm{d} x =\int_{I_{i,\delta}}\int_{I_{j,\delta}}\left(u(x)-u(y)\right)^2 \mathrm{d} x \mathrm{d} y.
\end{equation*}
We can estimate both the left and the right-hand sides of (\ref{eq:prop_cnpi}) by using the quantities $A_{i,j}$. 

{\bf Step 1:}  $\mathcal{H}_{21}(u)$ bound.- We start working on the term $\mathcal{H}_{21}(u)(\tau)$, where $0<\delta<y<x$. By swapping $(x,y)$ in the domain of integration, we get
\begin{align*}
\mathcal{H}_{21}(u)=&\int_{\delta}^{\infty}\int_{\delta}^{x}P_{\infty}(x)P_{\infty}(y)\left( u(x) - u(y) \right)^2 \mathrm{d} y \mathrm{d} x \\
 \le& \sum_{i=0}^{\infty}\sum_{j=0}^{i}\int_{I_{i,\delta}}\int_{I_{j,\delta}}\left(u(x)-u(y)\right)^2P_{\infty}(x)P_{\infty}(y) \mathrm{d} y \mathrm{d} x.
\end{align*}
Now, using the inequality (\ref{eq:Pinf_boundary}) and the symmetry $A_{i,j}=A_{j,i}$, we obtain
\begin{align}
\mathcal{H}_{21}(u) &\le B(\delta) ^2 \sum_{i=0}^{\infty}\sum_{j=0}^{i} p_i p_j\int_{I_{i,\delta}}\int_{I_{j,\delta}}\!\!\!\!\left(u(x)-u(y)\right)^2 \mathrm{d} y \mathrm{d} x
\nonumber\\
&= B(\delta) ^2 \sum_{i=0}^{\infty}\sum_{j=0}^{i} p_i p_j A_{i,j}= B(\delta) ^2 \sum_{j=0}^{\infty}\sum_{i=j}^{\infty} p_i p_j A_{i,j}= B(\delta) ^2 \sum_{i=0}^{\infty}\sum_{j=i}^{\infty} p_i p_j A_{i,j}.
\label{eq:proof_step1}
\end{align}

Note that some terms in this expression already appear in the right hand side of (\ref{eq:prop_cnpi}), since:
\begin{align}\label{eq:proof_step2}
\sum_{i=0}^{\infty} p_i^2 A_{i,i}=&\, \sum_{i=0}^{\infty} p_i p_i \int_{I_{i,\delta}}\int_{I_{i,\delta}}\left(u(x)-u(y)\right)^2 \mathrm{d} x \mathrm{d} y \nonumber \\
\le&\, \dfrac{1}{A(\delta)^2}\sum_{i=0}^{\infty} \int_{I_{i,\delta}}\int_{I_{i,\delta}}\left(u(x)-u(y)\right)^2 P_{\infty}(x)P_{\infty}(y) \mathrm{d} x \mathrm{d} y \nonumber \\
=&\, \dfrac{2}{A(\delta)^2}\sum_{i=0}^{\infty} \int_{I_{i,\delta}}\int_{\underset{x>y}{x \in I_{i,\delta}}}\left(u(x)-u(y)\right)^2 P_{\infty}(x)P_{\infty}(y) \mathrm{d} x \mathrm{d} y \nonumber \\
\le&\, \dfrac{2}{A(\delta)^2}\sum_{i=0}^{\infty} \int_{I_{i,\delta}}\int_{y}^{y+1}\left(u(x)-u(y)\right)^2 P_{\infty}(x)P_{\infty}(y) \mathrm{d} x \mathrm{d} y \nonumber \\
=&\, \dfrac{2}{A(\delta)^2} \int_{\delta}^{\infty}\int_{y}^{y+1}\left(u(x)-u(y)\right)^2 P_{\infty}(x)P_{\infty}(y) \mathrm{d} x \mathrm{d} y \nonumber \\
\le &\, \dfrac{P_M}{A(\delta)^2} \int_{\delta}^{\infty} \int_{y}^{y+1}\left(u(x)-u(y)\right)^2 P_{\infty}(y) \mathrm{d} x \mathrm{d} y \le \dfrac{ P_M}{A(\delta)^2} D(u),
\end{align}
where $P_M=\underset{x \in [\delta, \ \infty)}{\max} P_{\infty}(x) < \infty$ due to the properties described in Section \ref{sec:equilibrium_prop}.

In order to estimate $A_{i,j}$ for $j>i$ we fix $i,j$ and call $n:=j-i\ge 1$. We use $n-1$ ``intermediate reactions'' to write the following: introduce $n-1$ dummy integration variables $z_{i+1}, \dots, z_{j-1}$ and denote averaged integrals with a stroke. Thus, we have: 
\begin{align*}
4A_{i,j}=&\, \dashint_{I_{i,\delta}}\dashint_{I_{j,\delta}}\left(u(x)-u(y)\right)^2 \mathrm{d} x \mathrm{d} y \\
=&\, \dashint_{I_{i,\delta}}\dashint_{I_{i+1,\delta}} \cdots \dashint_{I_{j,\delta}}\left(u(x)-u(y)\right)^2 \mathrm{d} x\, \mathrm{d} z_{j-1} \cdots \mathrm{d} z_{i+1} \, \mathrm{d} y \\
=&\, \dashint_{I_{i,\delta}}\dashint_{I_{i+1,\delta}} \cdots \dashint_{I_{j,\delta}}\left(u(z_j)-u(z_i)\right)^2 \mathrm{d} z_{j}\, \mathrm{d} z_{j-1} \cdots \mathrm{d} z_{i}, 
\end{align*}
where the last step is just renaming $x\equiv z_j$ and $y\equiv z_i$. Observe that nothing has been done in the case $j=i+1$. Using the Cauchy-Schwarz inequality and (\ref{eq:def_Mj}), we have
\begin{align*}
	4A_{i,j}=&\, \dashint_{I_{i,\delta}}\dashint_{I_{i+1,\delta}} \cdots \dashint_{I_{j,\delta}}\left(\sum_{k=i}^{j-1}(u(z_{k+1})-u(z_k))\right)^2 \mathrm{d} z_{j}\, \mathrm{d} z_{j-1} \cdots \mathrm{d} z_{i} \\
	 \le&\, \dashint_{I_{i,\delta}}\dashint_{I_{i+1,\delta}} \cdots \dashint_{I_{j,\delta}}\left(\sum_{k=i}^{j-1}(u(z_{k+1})-u(z_k))^2m_k\right)\left(\sum_{k=i}^{j-1}\dfrac{1}{m_k}\right) \mathrm{d} z_{j}\, \mathrm{d} z_{j-1} \cdots \mathrm{d} z_{i} \\
	\le&\, M_j\dashint_{I_{i,\delta}}\dashint_{I_{i+1,\delta}} \cdots \dashint_{I_{j,\delta}}\left(\sum_{k=i}^{j-1}(u(z_{k+1})-u(z_k))^2m_k\right) \mathrm{d} z_{j}\, \mathrm{d} z_{j-1} \cdots \mathrm{d} z_{i} \\
	=&\, M_j\sum_{k=i}^{j-1}m_k\dashint_{I_{i,\delta}}\dashint_{I_{i+1,\delta}} \cdots \dashint_{I_{j,\delta}}\left(u(z_{k+1})-u(z_k)\right)^2 \mathrm{d} z_{j}\, \mathrm{d} z_{j-1} \cdots \mathrm{d} z_{i}\\
	=&\, M_j\sum_{k=i}^{j-1}m_k\dashint_{I_{k,\delta}}\dashint_{I_{k+1,\delta}} \left(u(z_{k+1})-u(z_k)\right)^2 \mathrm{d} z_{k+1}\, \mathrm{d} z_{k} = 4 M_j\sum_{k=i}^{j-1}m_k A_{k,k+1}. 
\end{align*}
Hence, we deduce that
\begin{equation*}
A_{i,j} \le M_j\sum_{k=i}^{j-1}m_k A_{k,k+1} \qquad \text{for all $j>i$.}
\end{equation*}
Thus, we get
\begin{align*}
	\sum_{i=0}^{\infty}\sum_{j=i+1}^{\infty}p_i p_j A_{i,j} \le&\, \sum_{i=0}^{\infty}\sum_{j=i+1}^{\infty}p_i p_j M_j\sum_{k=i}^{j-1}m_k A_{k,k+1}\\
	=&\, \sum_{k=0}^{\infty} m_k A_{k,k+1}\sum_{j=k+1}^{\infty} p_j M_j\sum_{i=0}^{k}p_i \le C_{\delta}^1\sum_{k=0}^{\infty} A_{k,k+1} m_k \sum_{j=k+1}^{\infty} M_j p_j.
\end{align*}
The inequality $\sum_{i=0}^{k}p_i\le C$, in the previous expression, holds because $\sum_{i=0}^{\infty}p_i$ is a convergent series due to the d'Alembert's ratio test. Moreover, 
\eqref{xxx} implies
\begin{equation}\label{eq:proof_step3}
\sum_{i=0}^{\infty}\sum_{j=i+1}^{\infty}p_i p_j A_{i,j} \le C \sum_{k=0}^{\infty} A_{k,k+1} p_k \le C D(u)
\end{equation}
for a generic constant $C>0$.
We finally work in the equation (\ref{eq:proof_step3}) to obtain
\begin{align*}
\sum_{k=0}^{\infty} A_{k,k+1} p_k =&\, \sum_{k=0}^{\infty} \int_{I_{k,\delta}}\int_{I_{k+1,\delta}}\left(u(x)-u(y)\right)^2  \mathrm{d} x \, p_k \, \mathrm{d} y\\
 \le &\, \frac{1}{A(\delta)}\sum_{k=0}^{\infty} \int_{I_{k,\delta}}\int_y^{y+1}\left(u(x)-u(y)\right)^2  \mathrm{d} x P_{\infty}(y)\mathrm{d} y\\
\le &\, \frac{1}{A(\delta)} \int_0^{\infty}\int_y^{y+1}\left(u(x)-u(y)\right)^2 P_{\infty}(y) \mathrm{d} x \mathrm{d} y=\frac{1}{A(\delta)} D(u),
\end{align*}
where we use that $y<\delta + \frac{k+1}{2}<\delta + \frac{k+2}{2} < y+1$ and (\ref{eq:Pinf_boundary}). We conclude by plugging the above estimate in (\ref{eq:proof_step3}), which together with equations (\ref{eq:proof_step1}) and (\ref{eq:proof_step2}) show that
\begin{equation}\label{eq:proof_step1_H21}
\lambda_1\mathcal{H}_{21}(u) \le D(u)\,,
\end{equation} 
for some constant $\lambda_1>0$.

{\bf Step 2:}  $\mathcal{H}_{22}(u)$ bound.-
To prove that there exists $\lambda_2>0$ such that
\begin{equation*}
\lambda_2 \mathcal{H}_{22}(u) \le \int_{0}^{\infty}\int_{y}^{y+1}P_{\infty}(y)\left( u(x) - u(y) \right)^2 \mathrm{d} x \mathrm{d} y ,
\end{equation*}
we use an intermediate variable $z\in(\delta, \ 1)$ as follows:
\begin{align*}
\int_{0}^{\delta}\int_{y}^{\infty}\left( u(x) - u(y) \right)^2 P_{\infty}(x)P_{\infty}(y)\mathrm{d} x \mathrm{d} y 
=& \, \dashint_{\delta}^1\int_{0}^{\delta}\int_{y}^{\infty}\left( u(x) - u(y) \right)^2 P_{\infty}(x)P_{\infty}(y)\mathrm{d} x \, \mathrm{d} y \, \mathrm{d} z \\
\le& \, 2\dashint_{\delta}^1\int_{0}^{\delta}\int_{y}^{\infty}\left( u(x) - u(z) \right)^2 P_{\infty}(x)P_{\infty}(y)\mathrm{d} x \, \mathrm{d} y \, \mathrm{d} z \\ &+2\dashint_{\delta}^1\int_{0}^{\delta}\int_{y}^{\infty}\left( u(z) - u(y) \right)^2 P_{\infty}(x)P_{\infty}(y)\mathrm{d} x \, \mathrm{d} y \, \mathrm{d} z \\
:= & \, 2I_1 +2I_2
\end{align*}
We bound each of the terms $I_1, \ I_2$. First, for $I_1$ we deduce that
\begin{align*}
I_1=&\dashint_{\delta}^1\int_{0}^{\delta}\int_{y}^{\infty}\left( u(x) - u(z) \right)^2 P_{\infty}(x)P_{\infty}(y)\mathrm{d} x \, \mathrm{d} y \, \mathrm{d} z 
\le\, \dashint_{\delta}^1\int_{0}^{\infty}\left( u(x) - u(z) \right)^2 P_{\infty}(x) \mathrm{d} x \, \mathrm{d} z \\
=&\, \dashint_{\delta}^1\int_{\delta}^{\infty}\left( u(x) - u(z) \right)^2 P_{\infty}(x) \mathrm{d} x \, \mathrm{d} z + \dashint_{\delta}^1\int_{0}^{\delta}\left( u(x) - u(z) \right)^2 P_{\infty}(x) \mathrm{d} x \, \mathrm{d} z
 := I_{11}+I_{12},
\end{align*}
since $\int_{0}^{\infty}P_{\infty}(y)\mathrm{d} y=1$.
For $I_{11}$ we use that $P_\infty$ is bounded below on $[\delta, \ 1]$ ($\frac{1}{C_{\delta}}\le P_\infty(x), \ x \in [\delta, \ 1]$) to deduce
\begin{align*}
I_{11}= &	\dashint_{\delta}^1\int_{\delta}^{\infty}\left( u(x) - u(z) \right)^2 P_{\infty}(x) \mathrm{d} x \, \mathrm{d} z 
\le\, C_{\delta}\dashint_{\delta}^1\int_{\delta}^{\infty}\left( u(x) - u(z) \right)^2 P_{\infty}(x) P_{\infty}(z) \mathrm{d} x \, \mathrm{d} z \\
	\le&\, \dfrac{C_{\delta}}{1-\delta}\int_{\delta}^{\infty}\!\!\!\!\int_{\delta}^{\infty}\left( u(x) - u(z) \right)^2 P_{\infty}(x) P_{\infty}(z) \mathrm{d} x \, \mathrm{d} z
	 =\, \dfrac{2C_{\delta}}{1-\delta}\int_{\delta}^{\infty}\!\!\!\!\int_{z}^{\infty}\left( u(x) - u(z) \right)^2 P_{\infty}(x) P_{\infty}(z) \mathrm{d} x \, \mathrm{d} z,
\end{align*}
Note that the right hand side of the above equation is bounded by a multiple of the term $\mathcal{H}_{21}(u)$, thus leading to $I_{11}\le C\mathcal{H}_{21}(u)$ with $C=\dfrac{2C_{\delta}}{1-\delta}$. Using (\ref{eq:proof_step1_H21}) we deduce that $I_{11}\le CD(u)$.

The integral $I_{12}$ is clearly smaller than the right hand side of (\ref{eq:prop_cnpi}) since it involves a smaller domain of integration, indeed we obtain
\begin{align*}
I_{12}=&\,	\dashint_{\delta}^1\int_{0}^{\delta}\left( u(x) - u(z) \right)^2 P_{\infty}(x) \mathrm{d} x \, \mathrm{d} z
=\, \dashint_{\delta}^1\int_{0}^{\delta}\left( u(x) - u(z) \right)^2 P_{\infty}(z) \mathrm{d} z \, \mathrm{d} x \\
=&\, \dfrac{1}{1-\delta} \int_{0}^{\delta}\int_{\delta}^{1}\left( u(x) - u(z) \right)^2 P_{\infty}(z) \mathrm{d} x \, \mathrm{d} z 
\le\dfrac{1}{1-\delta}\int_{0}^{\delta}\int_{z}^{z+1}\left( u(x) - u(z) \right)^2 P_{\infty}(z) \mathrm{d} x \, \mathrm{d} z 
\le CD(u)\,,
\end{align*}
since $z<\delta<x<1<z+1$.
For $I_2(\tau)$, notice that
\begin{align*}
I_2 =&\, \dashint_{\delta}^1\int_{0}^{\delta}\left( u(z) - u(y) \right)^2 P_{\infty}(y)\left(\int_{y}^{\infty}P_{\infty}(x)\mathrm{d} x \right) \mathrm{d} y \, \mathrm{d} z 
\le \, \dashint_{\delta}^1\int_{0}^{\delta}\left( u(z) - u(y) \right)^2 P_{\infty}(y) \mathrm{d} y \, \mathrm{d} z = I_{12},
\end{align*}
and thus, we also deduce that $I_2\le CD(u)$. Putting together the estimates on $I_{11}$, $I_{12}$ and $I_2$, we conclude that
\begin{equation}\label{eq:proof_step2_H22}
\lambda_2\mathcal{H}_{22}(u) \le D(u)\,,
\end{equation}
for some $\lambda_2>0$. Finally, inequalities (\ref{eq:proof_step1_H21}) and (\ref{eq:proof_step2_H22}) together imply that $\lambda\mathcal{H}_{2}(u) \le D(u)$ concluding the proof.
\end{proof}

\begin{proposition}\label{prop:exp_step2}
	There exists $\alpha>0$ such that
	\begin{equation}\label{eq:prop_exp_conv}
	\alpha \mathcal{D}(u)\le \mathcal{D}_2(u)\,.
	\end{equation}
with $u=p/P_\infty$, for all $p \in L^1((0,+\infty)) \cap L^2((0,+\infty), P_\infty^{-1})$.
\end{proposition}
\begin{proof}
Note that, $y<x<y+1$ on the left hand side of (\ref{eq:prop_exp_conv}). Thus, we can bound the term $\omega(x-y)$ with $x \in [y, \ y+1]$. Since $\omega(x)$ is a decreasing function of $x$, then
\begin{equation*}
\omega (1)=\frac{1}{b} e ^{\frac{-1}{b}}\le \omega(x-y) \le \frac{1}{b}=\omega (0) \qquad \text{with $x \in [y, \ y+1]$ and $y \in \mathbb{R}^+$}.
\end{equation*}
Moreover, the term $c(x)$ is bounded, $\varepsilon \le c(x) \le 1 $ for all $x \in \mathbb{R}^+$. So that:
\begin{align*}
\int_{0}^{\infty}\int_{y}^{y+1}P_{\infty}(y) \left( u(x) - u(y) \right)^2 \mathrm{d} x \mathrm{d} y 
\le&\, \frac{b}{\varepsilon}e ^{\frac{1}{b}}\int_{0}^{\infty}\int_{y}^{y+1}\omega(x-y)c(y)P_{\infty}(y)\left( u(x) - u(y) \right)^2 \mathrm{d} x \mathrm{d} y  \\
  \le&\, \frac{b}{\varepsilon}e ^{\frac{1}{b}}\int_{0}^{\infty}\int_{y}^{\infty}\omega(x-y)c(y)P_{\infty}(y)\left( u(x) - u(y) \right)^2 \mathrm{d} x \mathrm{d} y\\
   =&\, \frac{b}{a\varepsilon}e ^{\frac{1}{b}} \mathcal{D}_2(u),
\end{align*} 
which proves the inequality (\ref{eq:prop_exp_conv}).
\end{proof}

\begin{proof}[Proof of Theorem \ref{thm:exp}]
Putting together \eqref{eq:prop_cnpi} and \eqref{eq:prop_exp_conv} from Propositions \ref{prop:exp_step1} and \ref{prop:exp_step2}, we deduce that the entropy-entropy production inequality \eqref{ineq:expGRE_D2} holds. Lemma \ref{l:D2eqH2} together with \eqref{ineq:expGRE_D2} finally implies \eqref{ineq:expGRE}.  
As consequence, we deduce the exponential convergence towards $P_{\infty}$ for all mild solutions of \eqref{eq:Fried_good}.
\end{proof}


\subsection{Numerical illustration of exponential convergence}

The entropy functional, $\mathcal{G}_2(u)(t)$, is represented in the plots B of Figures \ref{fig:case1}-\ref{fig:case5}, which address the five possible steady states plots A of Figures \ref{fig:case1}-\ref{fig:case5} (see also Figure \ref{fig:distribution_shapes}). For all cases, these functions are represented in a semi-logarithm scale to numerically validate the exponential convergence shown in the previous section.

\begin{figure}[H]
	\centering
	\begin{minipage}[t]{0.49\linewidth}
		\centering
		\textbf{A}\\
		\includegraphics[width=1\textwidth]{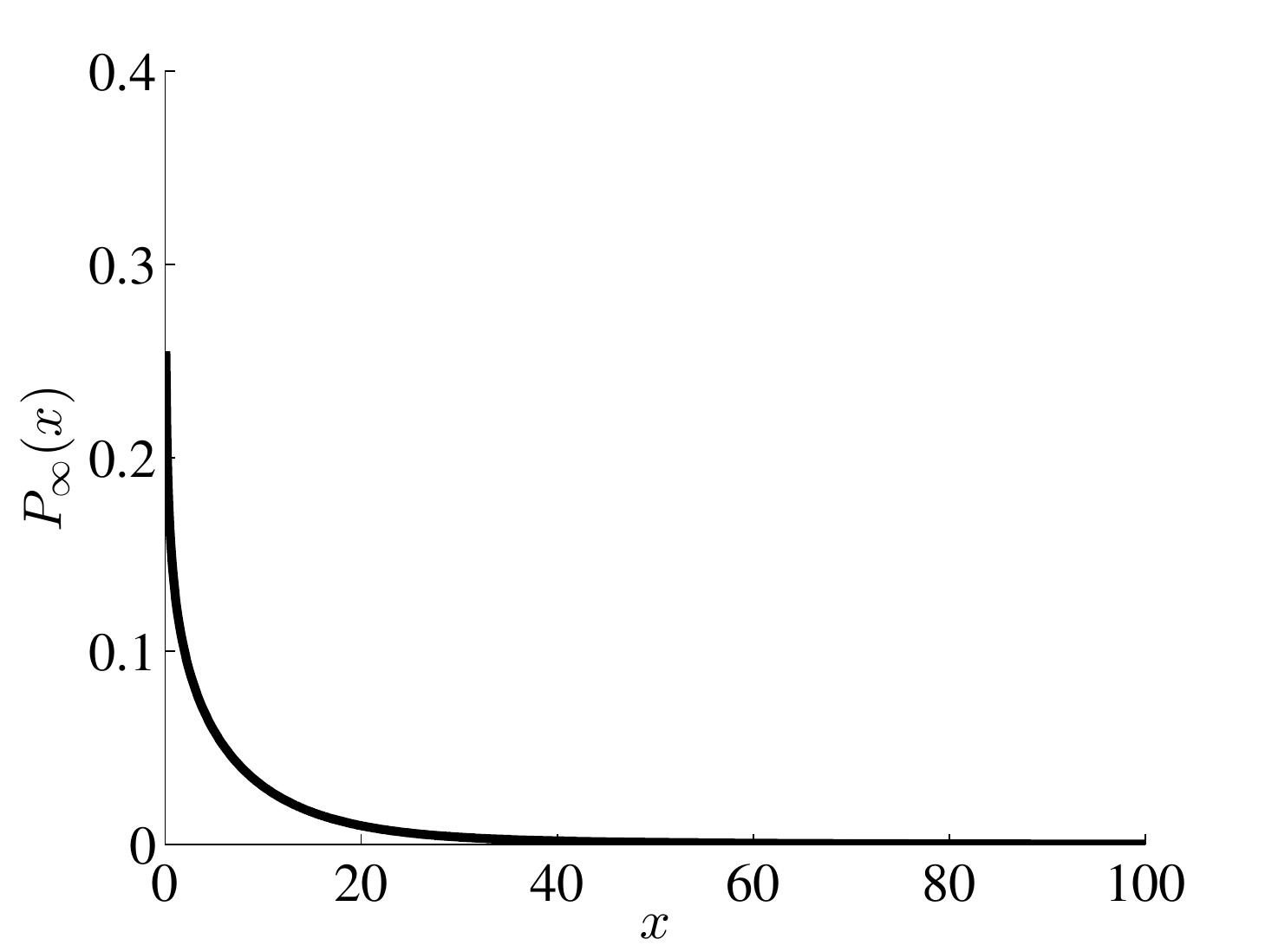}
	\end{minipage}
	\begin{minipage}[t]{0.49\linewidth}
		\centering
		\textbf{B}\\
		\includegraphics[width=1\textwidth]{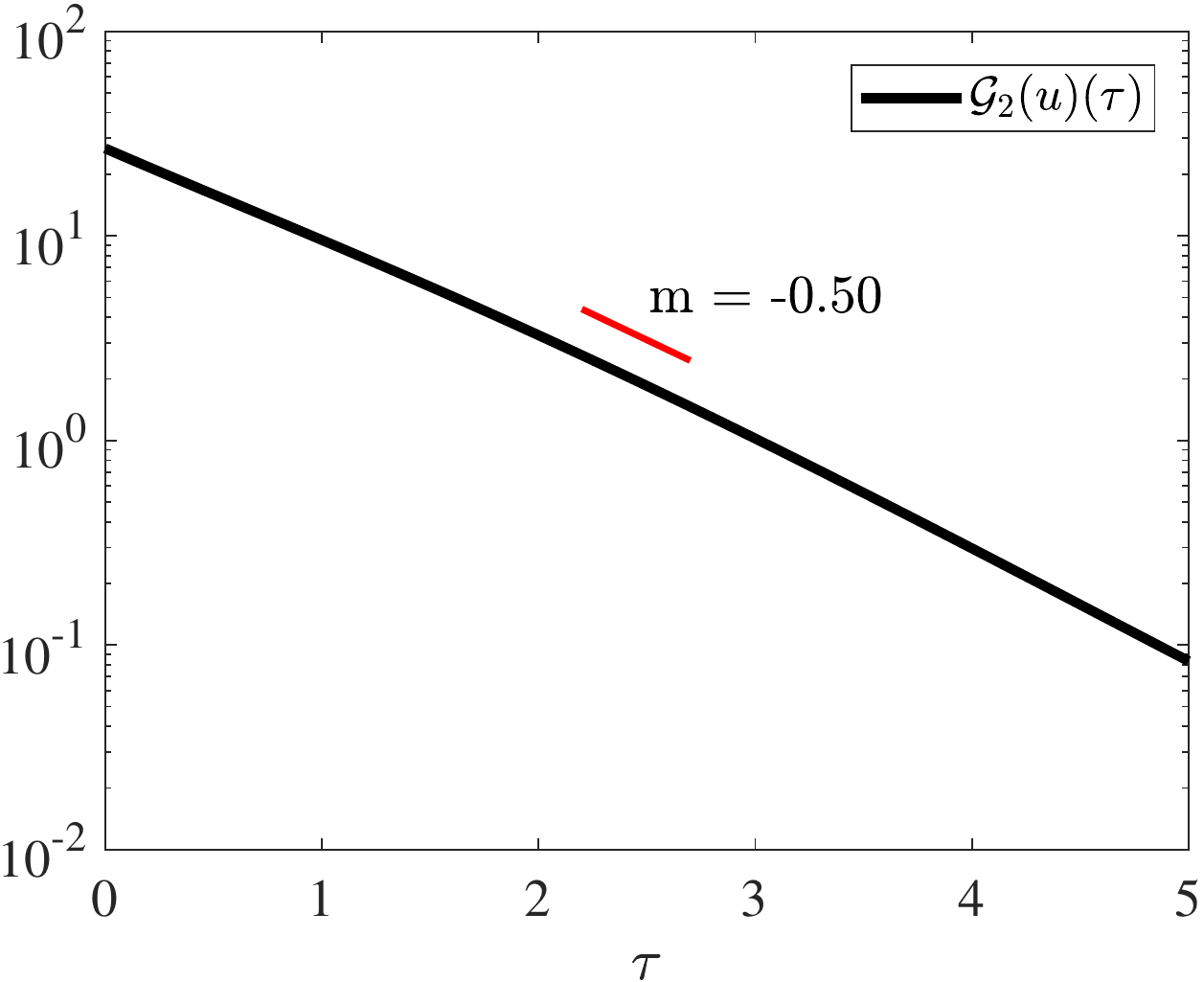}
	\end{minipage}
	\caption{Case 1 Fig \ref{fig:distribution_shapes}: $H=-4, \ \varepsilon=0.15, \ K=45, \ a=5, \ b=10$.}
	\label{fig:case1}
\end{figure}
\begin{figure}[H]
	\centering
	\begin{minipage}[t]{0.49\linewidth}
		\centering
		\textbf{A}\\
		\includegraphics[width=1\textwidth]{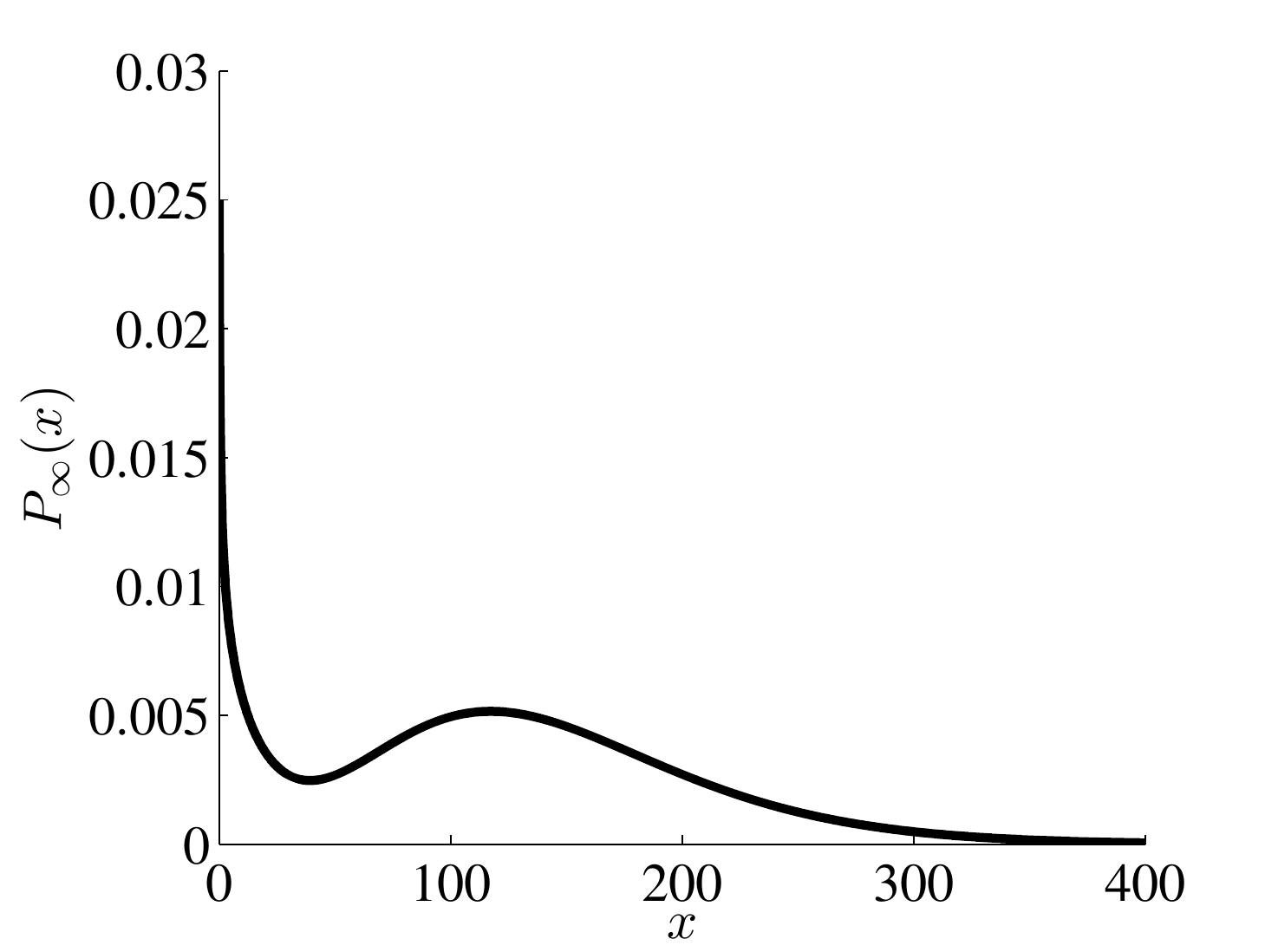}
	\end{minipage}
	\begin{minipage}[t]{0.49\linewidth}
		\centering
		\textbf{B}\\
		\includegraphics[width=1\textwidth]{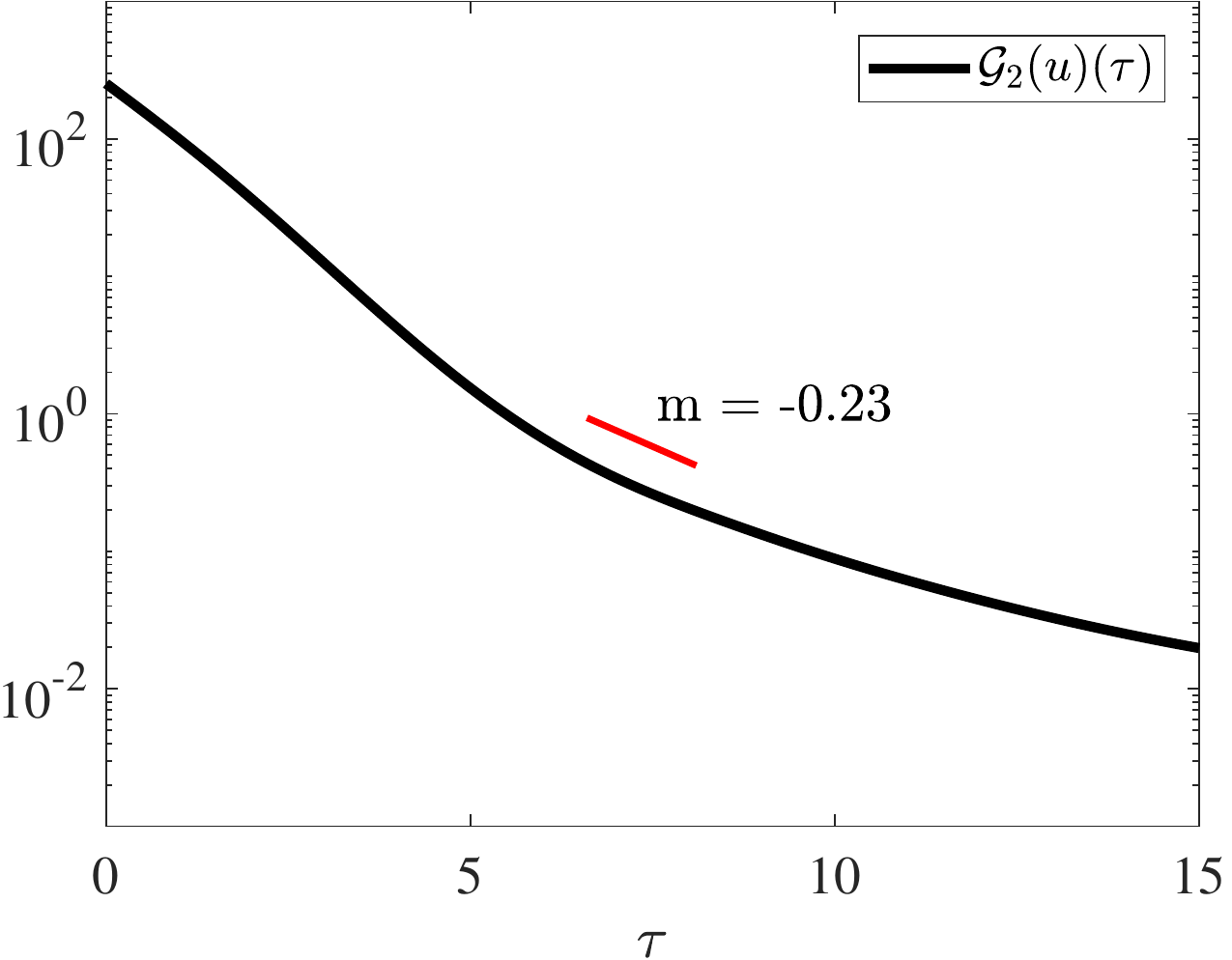}
	\end{minipage}
	\caption{Case 2 Fig \ref{fig:distribution_shapes}: $H=-4, \ \varepsilon=0.15, \ K=45, \ a=5, \ b=30$.}
	\label{fig:case2}
\end{figure}
\begin{figure}[H]
	\centering
	\begin{minipage}[t]{0.49\linewidth}
		\centering
		\textbf{A}\\
		\includegraphics[width=1\textwidth]{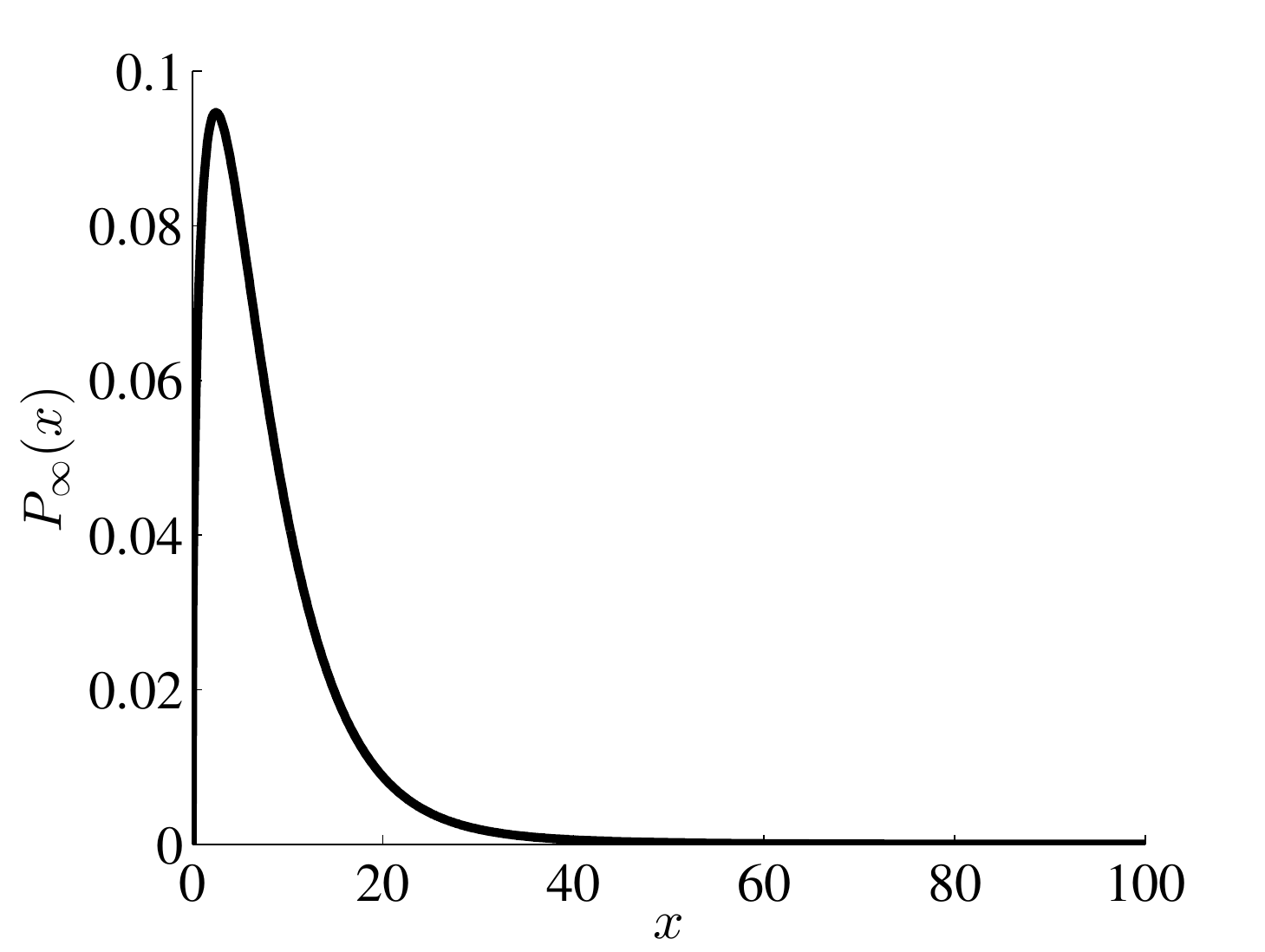}
	\end{minipage}
	\begin{minipage}[t]{0.49\linewidth}
		\centering
		\textbf{B}\\
		\includegraphics[width=1\textwidth]{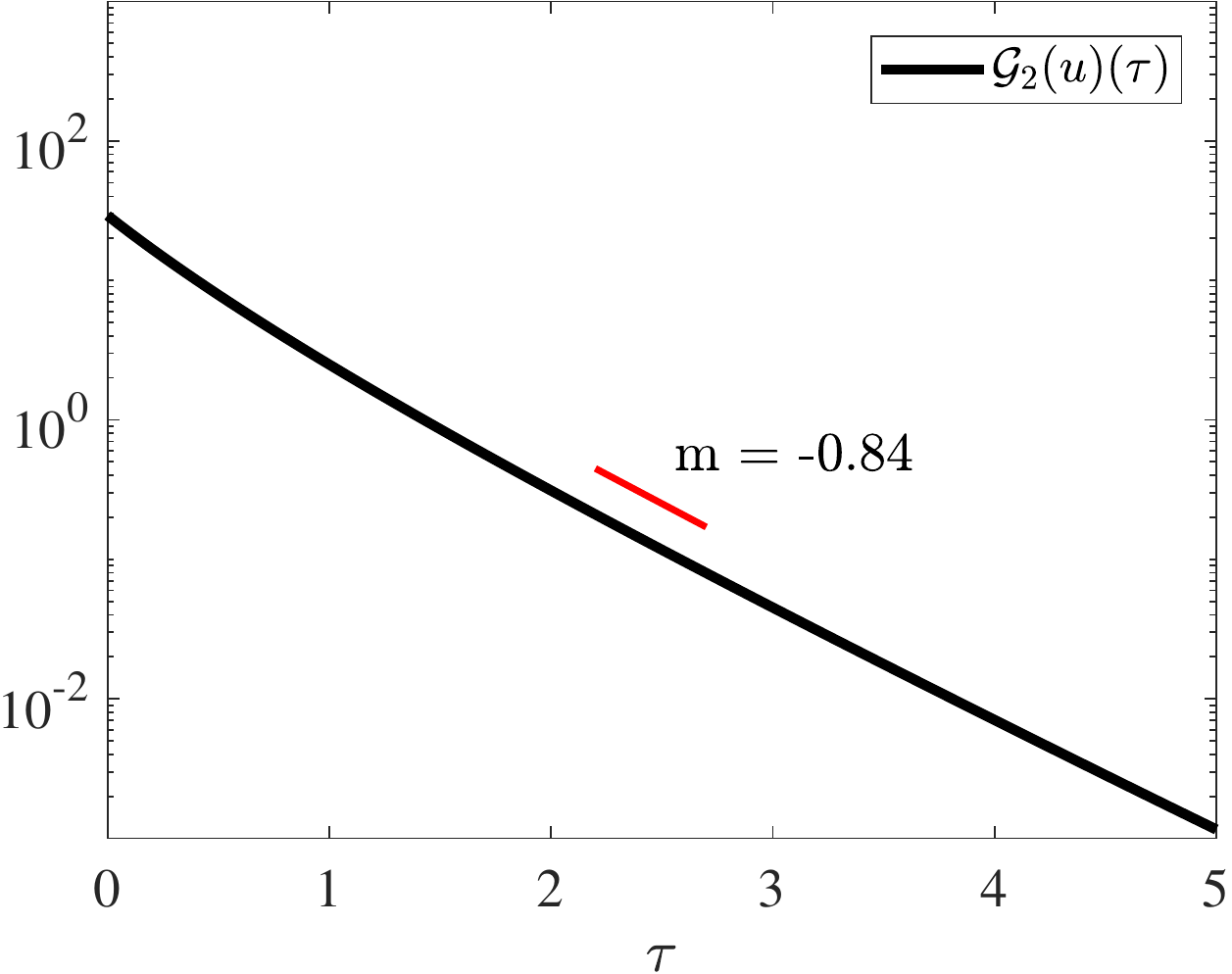}
	\end{minipage}
	\caption{Case 3 Fig \ref{fig:distribution_shapes}: $H=-4, \ \varepsilon=0.15, \ K=45, \ a=10, \ b=5$.}
	\label{fig:case3}
\end{figure}
\begin{figure}[H]
	\centering
	\begin{minipage}[t]{0.49\linewidth}
		\centering
		\textbf{A}\\
		\includegraphics[width=1\textwidth]{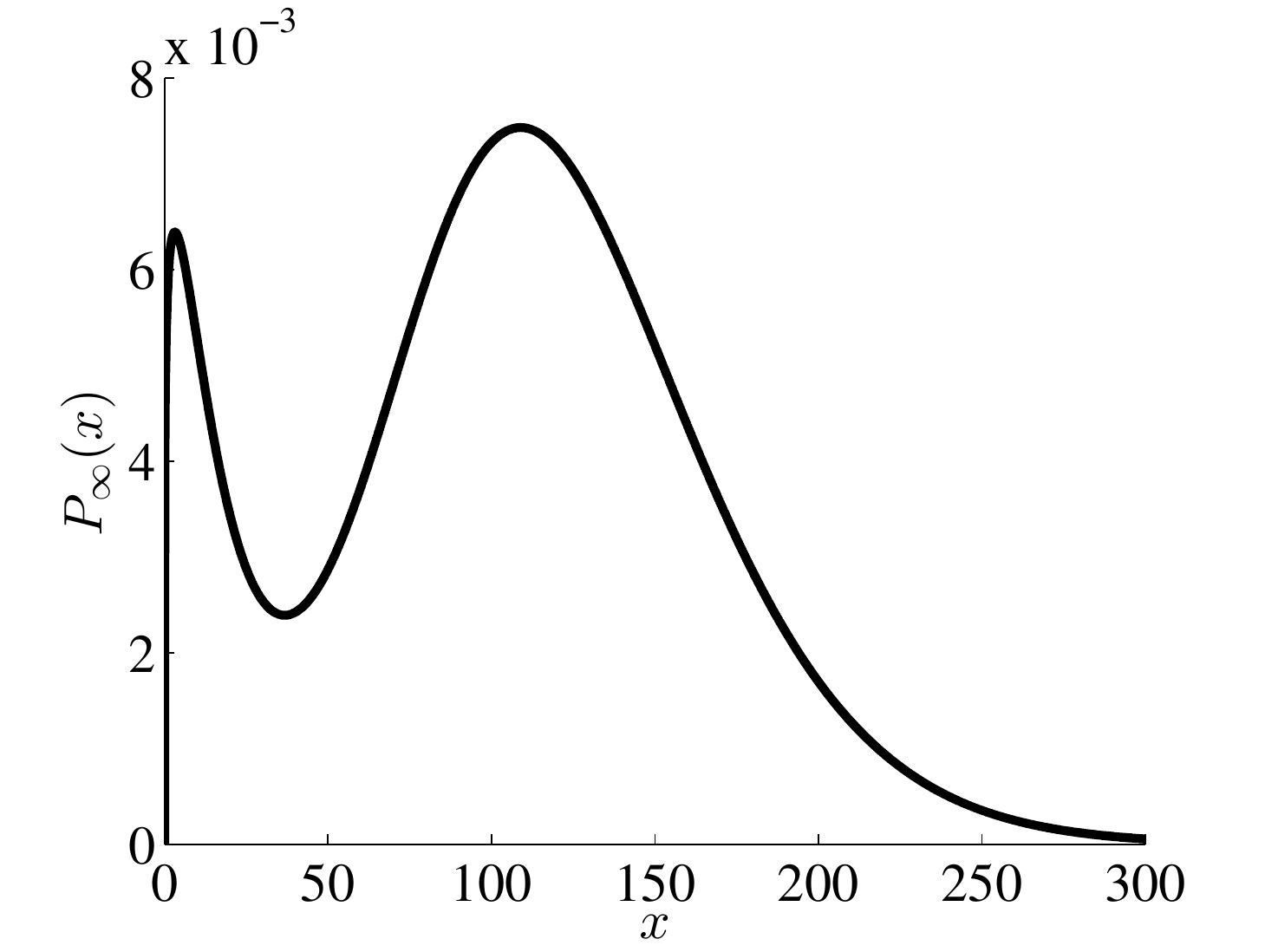}
	\end{minipage}
	\begin{minipage}[t]{0.49\linewidth}
		\centering
		\textbf{B}\\
		\includegraphics[width=1\textwidth]{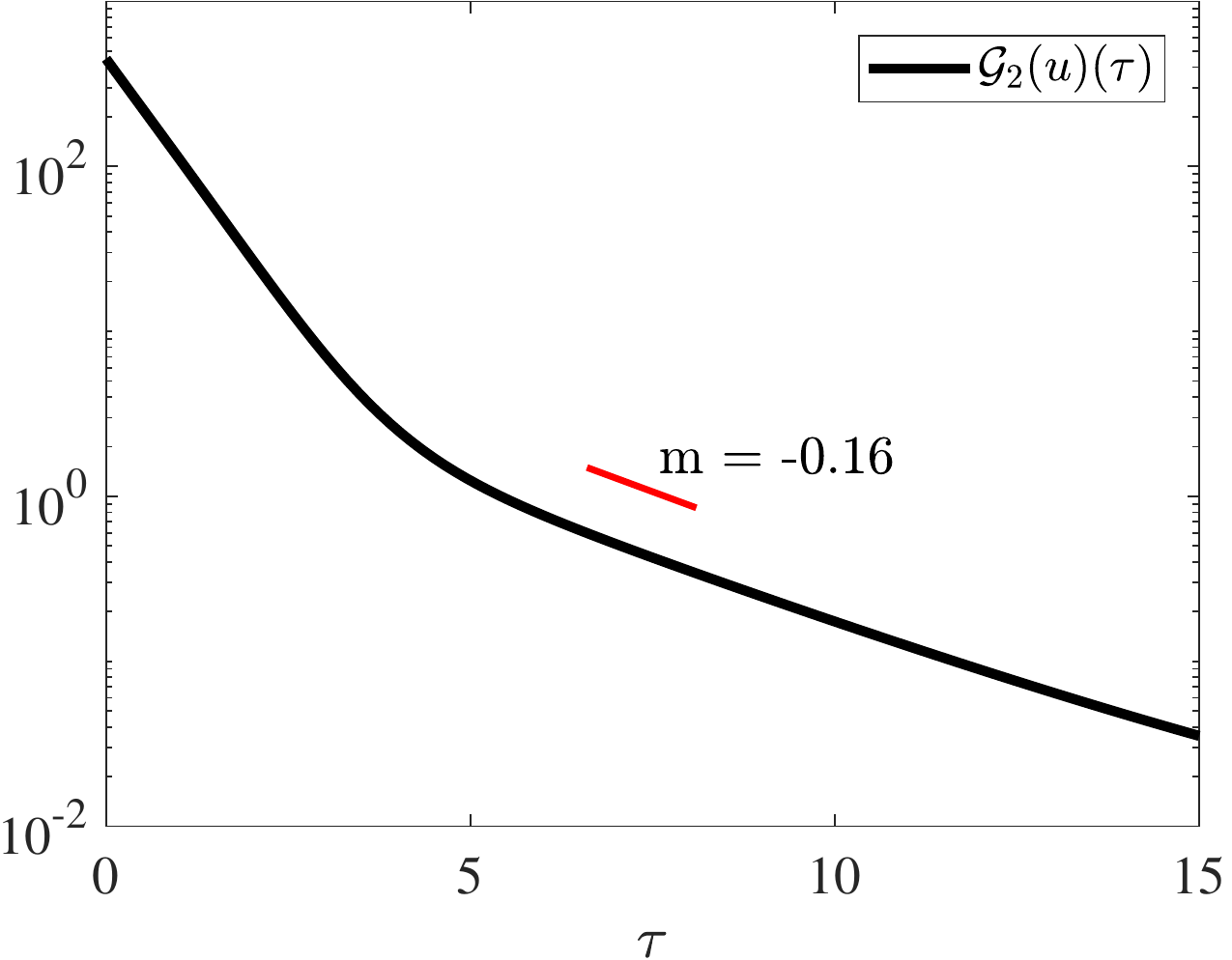}
	\end{minipage}
	\caption{Case 4 Fig \ref{fig:distribution_shapes}: $H=-4, \ \varepsilon=0.15, \ K=45, \ a=8, \ b=16$.}
	\label{fig:case4}
\end{figure}
\begin{figure}[H]
	\centering
	\begin{minipage}[t]{0.49\linewidth}
		\centering
		\textbf{A}\\
		\includegraphics[width=1\textwidth]{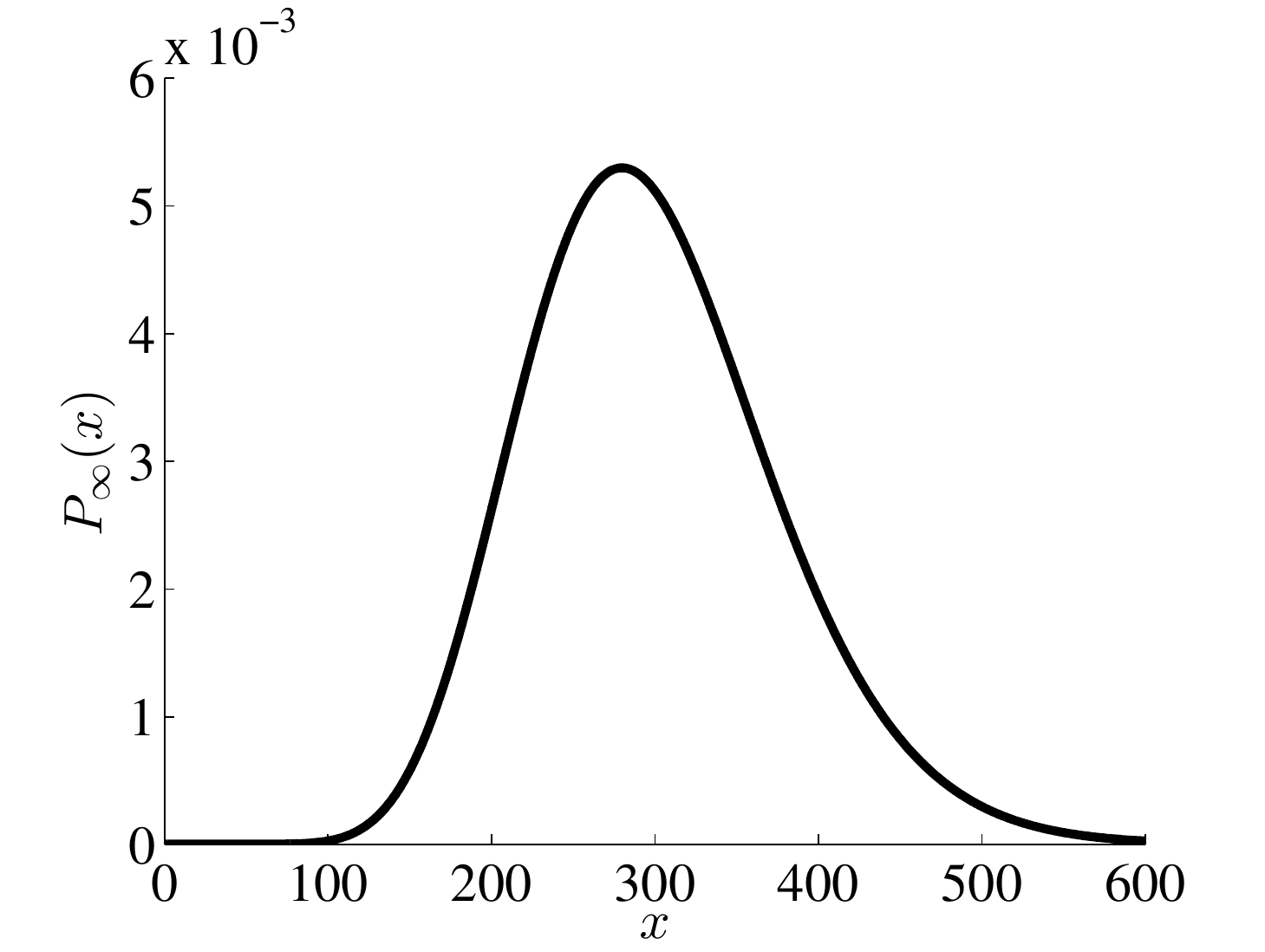}
	\end{minipage}
	\begin{minipage}[t]{0.49\linewidth}
		\centering
		\textbf{B}\\
		\includegraphics[width=1\textwidth]{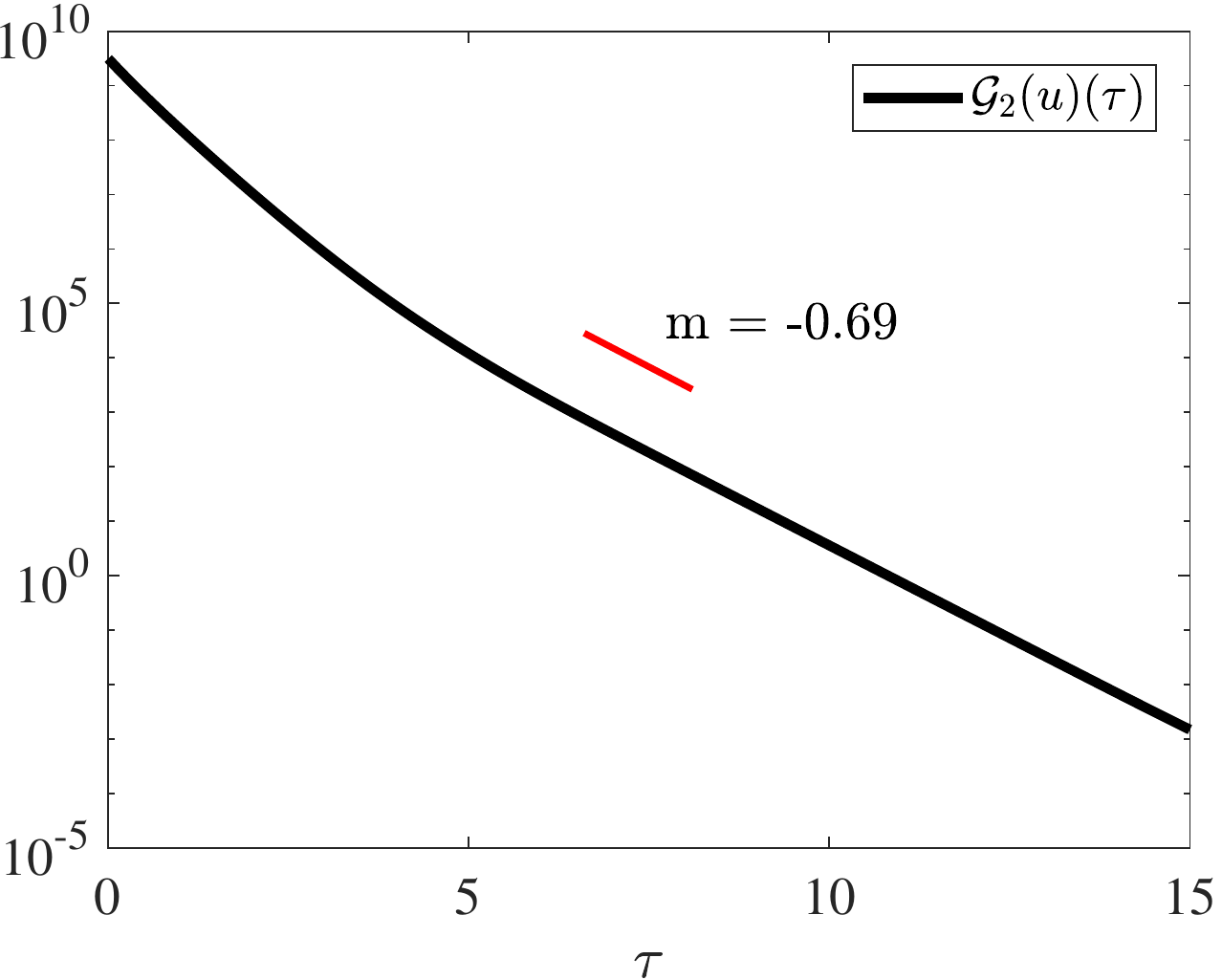}
	\end{minipage}
	\caption{Case 5 Fig \ref{fig:distribution_shapes}: $H=-4, \ \varepsilon=0.15, \ K=45, \ a=15, \ b=20$.}
	\label{fig:case5}
\end{figure}


\section{The $nD$ PIDE model}
\label{sec:nd}

We can generalise the entropy functional (\ref{eq:def_GRE}) defined
for the one dimension PIDE model in order to study the convergence of
the multidimensional model. A well-posedness theory of mild and
classical solutions satisfying the positivity and mass preservation,
the $L^1$-contraction principle, and the maximum principle can be
analogously obtained from the one dimensional strategy in Section
2. Let us summarize these properties in the next proposition.

\begin{proposition}
Given any classical solution of equation \eqref{eq:MEF_generalise} with normalised initial data, then the solution satisfies 
\begin{itemize}
\item[(i)] Mass conservation:
		\begin{equation*}
		\int_{{\mathbb{R}}_{+}^{n}}p(t,\mathbf{x})\mathrm{d} \mathbf{x}=\int_{{\mathbb{R}}_{+}^{n}}p_0(\mathbf{x})\mathrm{d} \mathbf{x} =1
		\end{equation*}

\item[(ii)] If $p_0$ is nonnegative, then the solution $p(t)$ of equation \eqref{eq:MEF_generalise} is nonnegative for all $t\geq 0$.

\item[(iii)] $L^1$-contraction principle:
		\begin{equation*}
		\int_{{\mathbb{R}}_{+}^{n}}|p(t,\mathbf{x})|\mathrm{d} \mathbf{x} \le \int_{{\mathbb{R}}_{+}^{n}}|p_0(\mathbf{x})|\mathrm{d} \mathbf{x}.
		\end{equation*}

\item[(iv)] $L^q$  bounds, $1<q<\infty$:
		\begin{equation*}
		\int_{{\mathbb{R}}_{+}^{n}}P_{\infty}(\mathbf{x})|u(t,\mathbf{x})|^q\mathrm{d} \mathbf{x} \le \int_{{\mathbb{R}}_{+}^{n}}P_{\infty}(\mathbf{x})|u_0(\mathbf{x})|^q\mathrm{d} \mathbf{x} ~~\text{with}~~ u(t,\mathbf{x}):=\frac{p(t,\mathbf{x})}{P_{\infty}(\mathbf{x})} \mbox{ and } u_0(\mathbf{x}):=\frac{p_0(\mathbf{x})}{P_{\infty}(\mathbf{x})}.
		\end{equation*}
		
\item[(v)] Maximum principle:
		\begin{equation*}
		\inf_{\mathbf{x} \in {\mathbb{R}}_{+}^{n}} u_0(\mathbf{x})\le u(t,\mathbf{x}) \le \sup_{\mathbf{x} \in {\mathbb{R}}_{+}^{n}} u_0(\mathbf{x}).
		\end{equation*}
	\end{itemize}
\end{proposition}
We will not do any details of these classical results. We just point
out that these properties can be formally seen as consequences of the
general relative entropy method \cite{Micheletal04,Micheletal05}.  Let
us now concentrate on the entropy method. Given $H(u)$ any convex
function of $u$, we define the $n$-dimensional general relative
entropy functional as:
\begin{equation*}
\mathcal{G}_H^n(u):= \int_{{\mathbb{R}}_{+}^{n}} H(u(\mathbf{x})) P_{\infty}(\mathbf{x}) \mathrm{d}\mathbf{x},
\end{equation*}
with $u(\mathbf{x}):=p(\mathbf{x})/P_\infty(\mathbf{x})$ as above. The main difference in the multidimensional case is that the stationary states are not explicit and thus, we need to assume certain properties on their behavior. In fact, in order to apply the entropy-entropy production method we make the following assumption:
\begin{assumption}\label{ass:ge}
	The following property holds
	\begin{equation*}
	\int_0^{\infty} \dfrac{\partial [H(u(\mathbf{x}))\gamma_x^i(\mathbf{x}) x_i P_\infty(\mathbf{x})]}{\partial x_i} \mathrm{d}x_i = 0 , \qquad \forall i=1,\cdots,n,
	\end{equation*}
	for any convex function $H(u)$ and for all $p \in L^1((0,+\infty)) \cap L^2((0,+\infty), P_\infty^{-1})$.
\end{assumption}

Similarly to the one dimensional case, we can obtain the following identity. The proof is totally analogous to the one of Lemma \ref{lemma:l2} and we skip it here for brevity.

\begin{lemma}\label{lemma:l2_nD}
	For any $i=1,\cdots,n$ the following equality is verified:
	\begin{align*}
	H'(u(\mathbf{x}))\dfrac{\partial [\gamma_x^i(\mathbf{x}) x_i p(\mathbf{x})]}{\partial x_i} =&\, \dfrac{\partial [H(u(\mathbf{x}))\gamma_x^i(\mathbf{x}) x_i P_{\infty}(\mathbf{x})]}{\partial x_i} \nonumber\\
	&+ \left(u(\mathbf{x})H'(u(\mathbf{x}))-H(u(\mathbf{x}))\right)\dfrac{\partial [\gamma_x^i(\mathbf{x}) x_i P_{\infty}(\mathbf{x})]}{\partial x_i}\,.
	\end{align*}
\end{lemma}

With this identy, we can now derive the evolution of the relative entropy as in the one dimensional case. We will not make explicit the time dependency of the solutions again for simplicity.

\begin{proposition}\label{prop:Hconvex_entropynD}
For any convex function $H(u(\mathbf{x}))$, the general entropy functional $\mathcal{G}_H^n(u)$ satisfies
	\begin{equation}\label{eq:Dt_Gentropy_eqn_nD}
	\dfrac{\mathrm{d}\mathcal{G}_H^n(u)}{\mathrm{d} t} = \sum_{i=1}^{n}k_m^i\! \int_{{\mathbb{R}}_{+}^{n}}\!\!\int_{y_i}^{\infty}\!\!\!\!\left[ H(u(\mathbf{x})) - H(u(\mathbf{y}_i)) +H'\left(u(\mathbf{x}))(u(\mathbf{y}_i)-u(\mathbf{x})\right)\right]
	\omega_{c,i} P_{\infty}(\mathbf{y}_i) \mathrm{d} x_i \mathrm{d} \mathbf{y}_i\leq 0\,,
	\end{equation}
with the shortcut $\omega_{c,i}=\omega_i(x_i-y_i) c_i(\mathbf{y}_i)$.
\end{proposition}

\begin{proof}[ Proof of proposition \ref{prop:Hconvex_entropynD}]
  We compute the time derivative of the general relative entropy
  functional to get
\begin{align*}
\dfrac{\mathrm{d}\mathcal{G}_H^n(u)}{\mathrm{d}t}=  \dfrac{\partial }{\partial t} \int_{{\mathbb{R}}_{+}^{n}}H(u(\mathbf{x}))P_{\infty}(\mathbf{x})\mathrm{d}\mathbf{x}& = \int_{{\mathbb{R}}_{+}^{n}}\dfrac{\partial }{\partial t} H(u(\mathbf{x}))P_{\infty}(\mathbf{x})\mathrm{d}\mathbf{x} = \int_{{\mathbb{R}}_{+}^{n}}H'(u(\mathbf{x}))\dfrac{\partial p }{\partial t} \mathrm{d}\mathbf{x} .
\end{align*}
Replacing the time derivative of $p(\mathbf{x})$ in the last equality by its expression (\ref{eq:MEF_generalise}), we obtain
\begin{align*}
\dfrac{\mathrm{d}\mathcal{G}_H^n(u)}{\mathrm{d} t}=&\, \int_{{\mathbb{R}}_{+}^{n}}H'(u(\mathbf{x}))\left( \sum_{i=1}^{n}\left( \dfrac{\partial}{\partial x_i}\left[\gamma_x^i(\mathbf{x}) x_i p(\mathbf{x})\right]\right) \right)\mathrm{d}\mathbf{x}\\
&\, +\int_{{\mathbb{R}}_{+}^{n}}H'(u(\mathbf{x}))\left( \sum_{i=1}^{n}\left( k_m^i \int_0^{x_i} \! \omega_i(x_i-y_i) c_i(\mathbf{y}_i)p(\mathbf{y}_i) \, \mathrm{d}y_i -k_1^ic_i(\mathbf{x})p(\mathbf{x})\right) \right)\mathrm{d}\mathbf{x}.
\end{align*}
Summations and integrals in the above expression are interchangeable, so that
\begin{align}\label{eq:n_derH_1_sumint}
\dfrac{\mathrm{d}\mathcal{G}_H^n(u)}{\mathrm{d} t}= &\,\sum_{i=1}^{n}\left( \int_{{\mathbb{R}}_{+}^{n}}H'(u(\mathbf{x}))\left( \dfrac{\partial}{\partial x_i}\left[\gamma_x^i(\mathbf{x}) x_i p(\mathbf{x})\right]\right) \right)\mathrm{d}\mathbf{x}  \nonumber\\
 &+\sum_{i=1}^{n}\left( \int_{{\mathbb{R}}_{+}^{n}}H'(u(\mathbf{x}))\left( k_m^i \int_0^{x_i} \! \omega_i(x_i-y_i) c_i(\mathbf{y}_i)p(\mathbf{y}_i) \, \mathrm{d}y_i -k_1^ic_i(\mathbf{x})p(\mathbf{x})\right) \right)\mathrm{d}\mathbf{x}.
\end{align}
Next, using Lemma \ref{lemma:l2_nD}, the first term on the right hand side in the above equation becomes
\begin{align}\label{eq:aux_npartial}
\sum_{i=1}^{n} \left( \int_{{\mathbb{R}}_{+}^{n}} \right. & \left. H'(u(\mathbf{x}))\left( \dfrac{\partial}{\partial x_i}\left[\gamma_x^i(\mathbf{x}) x_i p(\mathbf{x})\right]\right) \right)\mathrm{d}\mathbf{x} \nonumber\\
=&\sum_{i=1}^{n}\left( \int_{{\mathbb{R}}_{+}^{n}} \dfrac{\partial [H(u(\mathbf{x}))\gamma_x^i(\mathbf{x}) x_i P_{\infty}(\mathbf{x})]}{\partial x_i} + \left(u(\mathbf{x})H'(u(\mathbf{x}))-H(u(\mathbf{x}))\right)\dfrac{\partial [\gamma_x^i(\mathbf{x}) x_i P_{\infty}(\mathbf{x})]}{\partial x_i} \right)\mathrm{d}\mathbf{x} \nonumber \\
\hspace{-2cm}=&\sum_{i=1}^{n}\left( \int_{{\mathbb{R}}_{+}^{n}}  \left(u(\mathbf{x})H'(u(\mathbf{x}))-H(u(\mathbf{x}))\right)\dfrac{\partial [\gamma_x^i(\mathbf{x}) x_i P_{\infty}(\mathbf{x})]}{\partial x_i} \right)\mathrm{d}\mathbf{x} ,
\end{align}
this last identity holds using Assumption \ref{ass:ge}.
Note that, the first term in the last summation in equation (\ref{eq:aux_npartial}) is equivalent to
\begin{align}\label{eq:aux_npartial_derH}
\sum_{i=1}^{n}&\left( \int_{{\mathbb{R}}_{+}^{n}} u(\mathbf{x})H'(u(\mathbf{x}))\dfrac{\partial [\gamma_x^i(\mathbf{x}) x_i P_{\infty}(\mathbf{x})]}{\partial x_i} \right)\mathrm{d}\mathbf{x} \nonumber\\
&= \int_{{\mathbb{R}}_{+}^{n}} u(\mathbf{x})H'(u(\mathbf{x}))\sum_{i=1}^{n}\left( \dfrac{\partial [\gamma_x^i(\mathbf{x}) x_i P_{\infty}(\mathbf{x})]}{\partial x_i} \right)\mathrm{d}\mathbf{x} \nonumber\\
& =\int_{{\mathbb{R}}_{+}^{n}}u(\mathbf{x})H'(u(\mathbf{x}))\left( \sum_{i=1}^{n}\left(- k_m^i \int_0^{x_i} \! \omega_i(x_i-y_i) c_i(\mathbf{y}_i)P_{\infty}(\mathbf{y}_i) \, \mathrm{d}y_i + k_1^ic_i(\mathbf{x})P_{\infty}(\mathbf{x})\right) \right)\mathrm{d}\mathbf{x} \nonumber \\
& =\sum_{i=1}^{n}\left(\int_{{\mathbb{R}}_{+}^{n}}H'(u(\mathbf{x}))\left( - u(\mathbf{x})k_m^i \int_0^{x_i} \! \omega_i(x_i-y_i) c_i(\mathbf{y}_i)P_{\infty}(\mathbf{y}_i) \, \mathrm{d}y_i + k_1^ic_i(\mathbf{x})p(\mathbf{x})\right) \right)\mathrm{d}\mathbf{x},
\end{align}
and the second term in the last summation in equation (\ref{eq:aux_npartial}) is equivalent to
\begin{align}\label{eq:aux_npartial_H}
\sum_{i=1}^{n}&\left( \int_{{\mathbb{R}}_{+}^{n}} -H(u(\mathbf{x}))\dfrac{\partial [\gamma_x^i(\mathbf{x}) x_i P_{\infty}(\mathbf{x})]}{\partial x_i} \right)\mathrm{d}\mathbf{x} = \int_{{\mathbb{R}}_{+}^{n}}- H(u(\mathbf{x}))\sum_{i=1}^{n}\left( \dfrac{\partial [\gamma_x^i(\mathbf{x}) x_i P_{\infty}(\mathbf{x})]}{\partial x_i} \right)\mathrm{d}\mathbf{x}\nonumber \\
& =\int_{{\mathbb{R}}_{+}^{n}}-H(u(\mathbf{x}))\left( \sum_{i=1}^{n}\left(- k_m^i \int_0^{x_i} \! \omega_i(x_i-y_i) c_i(\mathbf{y}_i)P_{\infty}(\mathbf{y}_i) \, \mathrm{d}y_i + k_1^ic_i(\mathbf{x})P_{\infty}(\mathbf{x})\right) \right)\mathrm{d}\mathbf{x} \nonumber \\
& =\sum_{i=1}^{n}\left(\int_{{\mathbb{R}}_{+}^{n}}H(u(\mathbf{x}))\left(  k_m^i \int_0^{x_i} \! \omega_i(x_i-y_i) c_i(\mathbf{y}_i)P_{\infty}(\mathbf{y}_i) \, \mathrm{d}y_i - k_1^ic_i(\mathbf{x})P_{\infty}(\mathbf{x})\right) \right)\mathrm{d}\mathbf{x}.
\end{align}
Thus, using the expressions (\ref{eq:aux_npartial_derH})-(\ref{eq:aux_npartial_H}), replacing first in (\ref{eq:aux_npartial}) and finally in the equation (\ref{eq:n_derH_1_sumint}), we obtain the following equality
\begin{align*}
\dfrac{\mathrm{d}\mathcal{G}_H^n(u)}{\mathrm{d} t}=& \sum_{i=1}^{n}\left(\int_{{\mathbb{R}}_{+}^{n}}\left( - u(t,\mathbf{x})k_m^i \int_0^{x_i} \! \omega_i(x_i-y_i) c_i(\mathbf{y}_i)P_{\infty}(\mathbf{y}_i) \, \mathrm{d}y_i + k_1^ic_i(\mathbf{x})p(\mathbf{x})\right)  H'(u(\mathbf{x}))\right)\mathrm{d}\mathbf{x} \nonumber \\
&+ \sum_{i=1}^{n}\left(\int_{{\mathbb{R}}_{+}^{n}}H(u(\mathbf{x}))\left(  k_m^i \int_0^{x_i} \! \omega_i(x_i-y_i) c_i(\mathbf{y}_i)P_{\infty}(\mathbf{y}_i) \, \mathrm{d}y_i - k_1^ic_i(\mathbf{x})P_{\infty}(\mathbf{x})\right) \right)\mathrm{d}\mathbf{x} \nonumber \\
&+\sum_{i=1}^{n}\left( \int_{{\mathbb{R}}_{+}^{n}}H'(u(\mathbf{x}))\left( k_m^i \int_0^{x_i} \! \omega_i(x_i-y_i) c_i(\mathbf{y}_i)p(t,\mathbf{y}_i) \, \mathrm{d}y_i -k_1^ic_i(\mathbf{x})p(\mathbf{x})\right) \right)\mathrm{d}\mathbf{x} 
\end{align*}
\begin{align}\label{eq:n_derH_extended}
=&\, \sum_{i=1}^{n}\left(\int_{{\mathbb{R}}_{+}^{n}}H(u(\mathbf{x}))\left(  k_m^i \int_0^{x_i} \! \omega_i(x_i-y_i) c_i(\mathbf{y}_i)P_{\infty}(\mathbf{y}_i) \, \mathrm{d}y_i - k_1^ic_i(\mathbf{x})P_{\infty}(\mathbf{x})\right) \right)\mathrm{d}\mathbf{x} \nonumber \\
&+\sum_{i=1}^{n}\left( \int_{{\mathbb{R}}_{+}^{n}}\left( k_m^i \int_0^{x_i} \! \omega_i(x_i-y_i) c_i(\mathbf{y}_i)P_{\infty}(\mathbf{y}_i)\left[u(t,\mathbf{y}_i)-u(t,\mathbf{x})\right]\, \mathrm{d}y_i \right) H'(u(\mathbf{x}))\right)\mathrm{d}\mathbf{x}.
\end{align}
By changing the order of integration in the above expression and using the following identity
\begin{equation*}
\int_{y_i}^{\infty}\omega_i(x_i-y_i)\mathrm{d}x_i=1, \qquad \forall i=1,\cdots,n,
\end{equation*}
the equation (\ref{eq:n_derH_extended}) can be rewritten in the following equivalent form
\begin{align*}
\dfrac{\mathrm{d}\mathcal{G}_H^n(u)}{\mathrm{d} t}= &\, \sum_{i=1}^{n}\left(k_m^i\int_{{\mathbb{R}}_{+}^{n}}  \int_{y_i}^{\infty} \! \omega_i(x_i-y_i) c_i(\mathbf{y}_i)P_{\infty}(\mathbf{y}_i) \left[ H(u(\mathbf{x})) -H(u(t,\mathbf{y}_i))\right] \, \mathrm{d}x_i  \right)\mathrm{d}\mathbf{y}_i\\
 &\,+\sum_{i=1}^{n}\left( \int_{{\mathbb{R}}_{+}^{n}}\left( k_m^i \int_{y_i}^{\infty} \! \omega_i(x_i-y_i) c_i(\mathbf{y}_i)P_{\infty}(\mathbf{y}_i)H'(u(\mathbf{x}))\left[u(t,\mathbf{y}_i)-u(t,\mathbf{x})\right]\, \mathrm{d}x_i \right) \right)\mathrm{d}\mathbf{y}_i,
\end{align*}
which is equivalent to the expression (\ref{eq:Dt_Gentropy_eqn_nD}) defined in Proposition \ref{prop:Hconvex_entropynD}, thus concluding the derivation of the identity. Observe finally that due to the convexity of $H(u)$, we deduce that $H(u)-H(v)+H'(u)(v-u) \le 0$ for all $u,v$ leading to final claim.
\end{proof}     
      
As in the one dimensional case, we will focus on the $L^2$-relative entropy, i.e., we choose 
$H(u)=(u-1)^2$ to define
\begin{equation*}
\mathcal{G}_2^n(u):=\int_{{\mathbb{R}}_{+}^{n}}\left(u(\mathbf{x}) -1\right)^2 P_{\infty} \mathrm{d}\mathbf{x}
\end{equation*}
and
\begin{align*}
\mathcal{D}_2^n(u)= \sum_{i=1}^{n}k_m^i\int_{{\mathbb{R}}_{+}^{n}}\int_{y_i}^{\infty}\omega_i(x_i-y_i)\left[ u(\mathbf{x}) - u(\mathbf{y}_i) \right]^2 c_i(\mathbf{y}_i)P_\infty(\mathbf{y}_i) \mathrm{d} x_i \mathrm{d} \mathbf{y}_i .
\end{align*}
Proposition \ref{prop:Hconvex_entropynD} leads to the relation
  \begin{equation}\label{eq:n_Dt_entropy_eqn}
    \dfrac{\mathrm{d}\mathcal{G}_2^n(u)}{\mathrm{d} t}
    =-\mathcal{D}_2^n(u)\leq 0.
	\end{equation}


\subsection{Approach to equilibrium}

Based on the assumption \ref{ass:ge} on stationary solutions, we are now able to control the entropy by the entropy production except for a small error term.

\begin{lemma}
	\label{lem:D_almost_H}
	Assume that $p \leq C_1 P_\infty$ for some $C_1 > 0$. Then, for each
	$\epsilon > 0$ there exists a constant $K_\epsilon > 0$ depending on
	$C_1$ and $\epsilon$ such that:
	\begin{equation*}
	\mathcal{G}_2^n(u)
	\leq K_\epsilon \mathcal{D}_2^n(u)
	+ \epsilon\,.
	\end{equation*}
\end{lemma}

\begin{proof}
By expanding the square, we can write
\begin{equation}
\label{eq:almost_bound0}
\mathcal{G}_2^n(u) = 
\frac12
\int_{\R^n_+} \int_{\R^n_+}
P_\infty(\mathbf{x})
P_\infty(\mathbf{y})
(u(t,\mathbf{x}) - u(\mathbf{y}))^2
\d \mathbf{x} \d \mathbf{y}.
\end{equation}
We split the latter integral in two parts: the integral over
$\Omega_\delta \times \Omega_\delta$, and the integral over its
complement  with 
$$
\Omega_{\delta}=\overbrace{\left[\delta, \ 1/\delta\right] \times \dots \times \left[\delta, \ 1/\delta\right]}^{n \text{ times}} \text{ such that, }  \delta \in (0, \ 1). 
$$ 
For the integral over the complement, using $p \leq C_1 P_\infty$, we deduce
$$
\iint_{\R^{2n}_+ \setminus (\Omega_\delta \times \Omega_\delta)}
P_\infty(\mathbf{x})
P_\infty(\mathbf{y})
(u(\mathbf{x}) - u(\mathbf{y}))^2
\d \mathbf{x} \d \mathbf{y}
\leq
2 C_1^2  \iint_{\R^{2n}_+ \setminus (\Omega_\delta \times \Omega_\delta)}
P_\infty(\mathbf{x})
P_\infty(\mathbf{y})
\d \mathbf{x} \d \mathbf{y}.
$$
On the other hand, for the integral over
$\Omega_\delta \times \Omega_\delta$ we get
\begin{equation*}
\int_{\Omega_\delta}
\int_{\Omega_\delta}
P_\infty(\mathbf{x})
P_\infty(\mathbf{y})
(u(\mathbf{x}) - u(\mathbf{y}))^2
\d \mathbf{x} \d \mathbf{y}
\leq
K_{\delta,1} 
\int_{\Omega_\delta}
\int_{\Omega_\delta}
(u(\mathbf{x}) - u(\mathbf{y}))^2
\d \mathbf{x} \d \mathbf{y},
\end{equation*}
where
\begin{equation*}
K_{\delta,1} := \sup_{(\mathbf{x}, \mathbf{y}) \in \Omega_\delta
	\times \Omega_\delta} P_\infty(\mathbf{x}) P_\infty(\mathbf{y})
< +\infty.
\end{equation*}
We now rewrite $u(\mathbf{x}) - u(\mathbf{y})$ as a sum of $n$
terms, each of which being a difference of values of $u$ at points
which differ only by one coordinate
\begin{equation*}
u(\mathbf{x}) - u(\mathbf{y})
=
\sum_{i=1}^n
\Big(
u(x_1, \dots, x_i, y_{i+1}, \dots, y_n)
- u(x_1, \dots, x_{i-1}, y_{i}, \dots, y_n)
\Big),
\end{equation*}
(where it is understood that
$u(x_1, \dots, x_i, y_{i+1}, \dots, y_n) = u(\mathbf{x})$ for $i=n$,
and $u(x_1, \dots,$ $ x_{i-1}, y_{i}, \dots, y_n) = u(\mathbf{y})$ for
$i=1$). Then, by Cauchy-Schwarz's inequality we have
\begin{align*}
\int_{\Omega_\delta} \int_{\Omega_\delta}
(u(\mathbf{x}) - & u(\mathbf{y}))^2
\d \mathbf{x} \d \mathbf{y}
\\
&\leq
n
\sum_{i=1}^n
\int_{\Omega_\delta} \int_{\Omega_\delta}
\Big( u(x_1, \dots, x_i, y_{i+1}, \dots, y_n)
- u(x_1, \dots, x_{i-1}, y_{i}, \dots, y_n) \Big)^2
\d \mathbf{x} \d \mathbf{y}
\\
&=
n \left(\frac{1}{\delta} - \delta \right)^{n-1}
\sum_{i=1}^n
\int_{[\delta, 1/\delta]^n}
\int_{[\delta, 1/\delta]}
\Big( u(\mathbf{x})
- u(\mathbf{y}_i)
\Big)^2
\d x_i \d\mathbf{y}_i
\\
&=
2n  \left(\frac{1}{\delta} - \delta \right)^{n-1}
\sum_{i=1}^n
\int_{[\delta, 1/\delta]^n}
\int_{y_i}^{1/\delta}
\Big( u(\mathbf{x})
- u(\mathbf{y}_i)
\Big)^2
\d x_i \d\mathbf{y}_i
\\
&\leq
K_{\delta,2}
\sum_{i=1}^n
k_m^i
\int_{[\delta, 1/\delta]^n}
\int_{y_i}^{1/\delta}
\omega_i(x_i-y_i) c_i(\mathbf{y}_i) P_\infty(\mathbf{y}_i)
\Big( u(\mathbf{x})
- u(\mathbf{y}_i)
\Big)^2
\d x_i \d\mathbf{y}_i,
\end{align*}
therefore we conclude that
\begin{equation}
\label{eq:almost_bound1}
\int_{\Omega_\delta} \int_{\Omega_\delta}
(u(\mathbf{x}) - u(\mathbf{y}))^2
\d \mathbf{x} \d \mathbf{y}
\leq
K_{\delta,2} \mathcal{D}_2(p),
\end{equation}
where $K_{\delta,2}$ is defined by
\begin{equation*}
2 n  \left(\frac{1}{\delta} - \delta \right)^{n-1}  K_{\delta,2}^{-1}
= \inf \big( k_m^i \omega_i(x_i-y_i) c_i(\mathbf{y}_i)
P_\infty(\mathbf{y}_i) \big),
\end{equation*}
with the infimum running over all $i = 1, \dots, n$ and over all the
points in the domain of integration. We notice that the
first of the equalities in \eqref{eq:almost_bound1} is just obtained
by integrating in the variables that do not appear in the expression
and renaming the others; and the second equality is due to the
symmetry of the integrand in the variables $(x_i, y_i)$. Using
\eqref{eq:almost_bound0}--\eqref{eq:almost_bound1} finally gives:
\begin{equation*}
\mathcal{G}_2^n(u)
\leq
C_1^2  \iint_{\R^{2n}_+ \setminus (\Omega_\delta \times \Omega_\delta)}
P_\infty(\mathbf{x})
P_\infty(\mathbf{y})
\d \mathbf{x} \d \mathbf{y}
+
\frac12 K_{\delta,1} K_{\delta,2} \mathcal{D}_2^n(u).
\end{equation*}
We may choose $\delta > 0$ such that the first term is smaller than
$\epsilon$. This gives then the result with
$K_\epsilon = \frac12 K_{\delta,1} K_{\delta,2}$.
\end{proof}

\begin{theorem}[Long-time behaviour]
  Given any mild solution $p$ with normalised nonnegative initial data
  $p_0 \in L^1(\R_+)$ to equation \eqref{eq:MEF_generalise} and given
  a stationary solution $P_\infty(\mathbf{x})$ to
  \eqref{eq:MEF_generalise} satisfying assumption \ref{ass:ge}, then
	\begin{equation*}
	\lim\limits_{t \rightarrow \infty}
	\int_{{\mathbb{R}}_{+}^{n}}|p(t,\mathbf{x})-P_\infty(\mathbf{x})|^2
	\mathrm{d} \mathbf{x} = 0.
	\end{equation*}
        As a consequence, stationary solutions $P_\infty(\mathbf{x})$
        of \eqref{eq:MEF_generalise} satisfying assumption
        \ref{ass:ge}, if they exist, they are unique.
\end{theorem}

\begin{proof}
\textbf{Step 1: Proof for ``nice'' initial data.} We first prove the
result for initial data $p_0 \in L^1(\R_+)$ such that
$p_0 \leq C_1 P_\infty$, for some constant $C_1 > 0$. Observe that
this implies in particular that
$p_0 \in L^2(\R_+, P_\infty(\mathbf{x})^{-1} \d \mathbf{x})$. For
such initial data we deduce that for all $t \geq 0$
\begin{equation*}
p(t,\mathbf{x}) \leq C_1 P_\infty(\mathbf{x})
\quad \text{ for almost all $\mathbf{x} \in \R^n_+$}\,,
\end{equation*}
from the maximum principle. This enables us to use Lemma \ref{lem:D_almost_H}. Using
the general entropy identity with $H(u)=(u-1)^2$, from Proposition
\ref{prop:Hconvex_entropynD} we obtain:
\begin{equation}
\label{aux1_corol_nD}
\dfrac{\mathrm{d} \mathcal{G}_2^n(u)}{\mathrm{d} t}
=
-\mathcal{D}_2^n(u).
\end{equation}
Next, by using time integration on $[0,T]$ in equation (\ref{aux1_corol_nD}),
the following equality holds for all $T>0$:
\begin{align*}
\mathcal{G}_2^n(u)(T)
+
\int_0^T \mathcal{D}_2^n(p)(t) \d t
=   \mathcal{G}_2^n(u)(0),
\end{align*}
from which we deduce that:
\begin{equation}\label{eq:intD}
\int_0^\infty \mathcal{D}_2^n(u)(t) \d t < \infty.
\end{equation}
From (\ref{eq:intD}), there exists a sequence $(t_s)_{s \geq 1}$ such
that $\mathcal{D}_2^n(u)(t_s) \to 0$ as $s \to +\infty$. Thus if we take
any $\epsilon > 0$, then Lemma \ref{lem:D_almost_H} gives:
\begin{equation*}
\mathcal{G}_2^n(u)(t_s)
\leq K_\epsilon \mathcal{D}_2^n(u)(t_s) + \epsilon
\to \epsilon
\quad \text{as $s \to +\infty$}.
\end{equation*}
Since $\mathcal{G}_2^n(u)(t)$ is decreasing in $t$,
this shows that $\lim_{t \to +\infty} \mathcal{G}_2^n(u)(t) \leq \epsilon$. Since $\epsilon$ is arbitrary chosen, we deduce
that:
\begin{equation*}
\mathcal{G}_2^n(u)(t)
\to 0
\quad \text{as $t \to +\infty$}.
\end{equation*}
\textbf{Step 2: Proof for all integrable initial data.} It is now classical to extend the result in step 1 to all initial data in $L^1(\R^n_+)$ by the $L^1$-contraction principle. In fact, any
$p_0 \in L^1(\R^n_+)$ can be approximated in $L^1(\R^n_+)$ by a sequence $(p_0^s)_{s \geq 1}$ such that $p_0^s \leq s P_\infty$, for all $s \geq 1$.  Thus consider the solution $p^s$ associated to initial data $p_0^s$. By step 1, we get
\begin{equation*}
\int_0^\infty |p^s(t, \mathbf{x}) - P_\infty(\mathbf{x})| \d \mathbf{x}
\to 0
\quad \text{as $t \to +\infty$},
\end{equation*}
since
$\mathcal{G}_2^n(u_s)(t) \geq \| p^s(t, \mathbf{x}) - P_\infty(\mathbf{x})\|_1^2$ with $u^s=\frac{p^s}{P_\infty}$. Hence, for $s \geq 1$ we deduce
\begin{align*}
\int_0^\infty |p(t, \mathbf{x}) - P_\infty(\mathbf{x})| \d
\mathbf{x}&
\leq
\int_0^\infty |p(t, \mathbf{x}) - p^s(t, \mathbf{x})| \d \mathbf{x}
+ \int_0^\infty | p^s(t, \mathbf{x}) - P_\infty(\mathbf{x})| \d \mathbf{x}
\nonumber \\
&\leq
\int_0^\infty |p_0(\mathbf{x}) - p_0^s(\mathbf{x})| \d \mathbf{x}
+ \int_0^\infty | p^s(t, \mathbf{x}) - P_\infty(\mathbf{x})| \d \mathbf{x}\,,
\end{align*}
from the $L^1$-contraction principle. This easily leads to the result since
$$
\lim_{s \to \infty} \int_0^\infty |p_0(\mathbf{x}) - p_0^s(\mathbf{x})| \d \mathbf{x}=0
\qquad \mbox{and} \qquad
\lim_{t \to \infty} \int_0^\infty | p^s(t, \mathbf{x}) - P_\infty(\mathbf{x})| \d \mathbf{x}=0\,,
$$
for all $s\geq 1$.
\end{proof}


\subsection{Numerical exploration of the convergence rates}

The entropy functional, $\mathcal{G}_2^n(u)(t)$, is represented in the plots B of Figures
\ref{fig:EX2_2D}-\ref{fig:EX4_2D}, which address three possible steady
states (plots A of Figures \ref{fig:EX2_2D}-\ref{fig:EX4_2D}) that
have been obtained using the SELANSI toolboox \cite{pajaroetal17b}. For all cases, these functions are represented in a semi-logarithm scale to numerically check if the convergence
shown in the previous section is exponential in higher dimensions.

In the first example, Figure \ref{fig:EX2_2D}, we consider two
different self-regulated proteins with input functions:
\begin{equation*}
  c_i(x_i)=\frac{K_i^{H_i}
    + \varepsilon_i x_i^{H_i}}{K_i^{H_i} + x_i^{H_i}}
  \quad \text{for $i=1,2$},	
\end{equation*}
with $H_i=-4$, $\varepsilon_i=0.15$, $K_i=45$, $a_i=5$ and $b_i=10$ as
in the example depicted in Figure \ref{fig:case1}.

The second example, Figure \ref{fig:EX1_2D}, is a self and
cross-regulated gene network expressing two different proteins where
the first one activates the production of both itself and the second
protein, while the second protein inhibits the expression of both
proteins. The input functions considered, as in \cite{pajaroetal17},
read:
\begin{equation}\label{eq:hill_multivariate}
\begin{array}{rl}
c_1(\mathbf{x})= &   \dfrac{\epsilon_{11}x_1^{H_{11}}x_2^{H_{12}} + \epsilon_{12}K_{11}^{H_{11}}x_2^{H_{12}} + \epsilon_{13}x_1^{H_{11}}K_{12}^{H_{12}} + K_{11}^{H_{11}}K_{12}^{H_{12}}}{x_1^{H_{11}}x_2^{H_{12}} + K_{11}^{H_{11}}x_2^{H_{12}} + x_1^{H_{11}}K_{12}^{H_{12}} + K_{11}^{H_{11}}K_{12}^{H_{12}}},  \\
&\\
c_2(\mathbf{x})= &   \dfrac{\epsilon_{21}x_2^{H_{22}}x_1^{H_{21}} + \epsilon_{22}K_{22}^{H_{22}}x_1^{H_{21}} + \epsilon_{23}x_2^{H_{22}}K_{21}^{H_{21}} + K_{22}^{H_{22}}K_{21}^{H_{21}}}{x_2^{H_{22}}x_1^{H_{21}} + K_{22}^{H_{22}}x_1^{H_{21}} + x_2^{H_{22}}K_{21}^{H_{21}} + K_{22}^{H_{22}}K_{21}^{H_{21}}},
\end{array}
\end{equation}
with $H_{11}=-4$, $H_{21}=-6$, $H_{12}=H_{22}=2$, $K_{11}=K_{12}=45$, $K_{21}=K_{22}=70$, $\varepsilon_{11}=\varepsilon_{21}=0.002$, $\varepsilon_{12}=0.02$, $\varepsilon_{22}=0.1$, $\varepsilon_{13}=\varepsilon_{23}=0.2$ and network parameters $\gamma_x^1=\gamma_x^2=1$, $\gamma_m^1=\gamma_m^2=25$, $k_m^1=10$, $k_m^2=20$, $b_1=10$ and $b_2=20$.

Our third example, figure \ref{fig:EX4_2D}, corresponds to a mutual
repressing network of two genes in which the protein produced by the
expression of one gene inhibits the production of the other protein in
the network. The input functions, as in \cite{pajaroetal17}, for this
example take the following form:
\begin{equation}\label{eq:hill_cross_reg}
c_1(\mathbf{x})= \dfrac{K_1^{H_{12}}+\varepsilon_1 x_2^{H_{12}}}{K_1^{H_{12}}+x_2^{H_{12}}}, \qquad
c_2(\mathbf{x})= \dfrac{K_2^{H_{21}}+\varepsilon_2 x_1^{H_{21}}}{K_2^{H_{21}}+x_1^{H_{21}}},
\end{equation}
with $H_{12}=H_{21}=4$, $K_{1}=K_{2}=45$ and
$\varepsilon_{1}=\varepsilon_{2}=0.15$. The dimensionless network
parameters are $\gamma_x^1=\gamma_x^2=1$, $\gamma_m^1=\gamma_m^2=25$,
$k_m^1=k_m^2=8$ and $b_1=b_2=16$.

\begin{figure}[H]
	\centering
	\begin{minipage}[t]{0.49\linewidth}
		\centering
		\textbf{A}\\
		\includegraphics[width=1\textwidth]{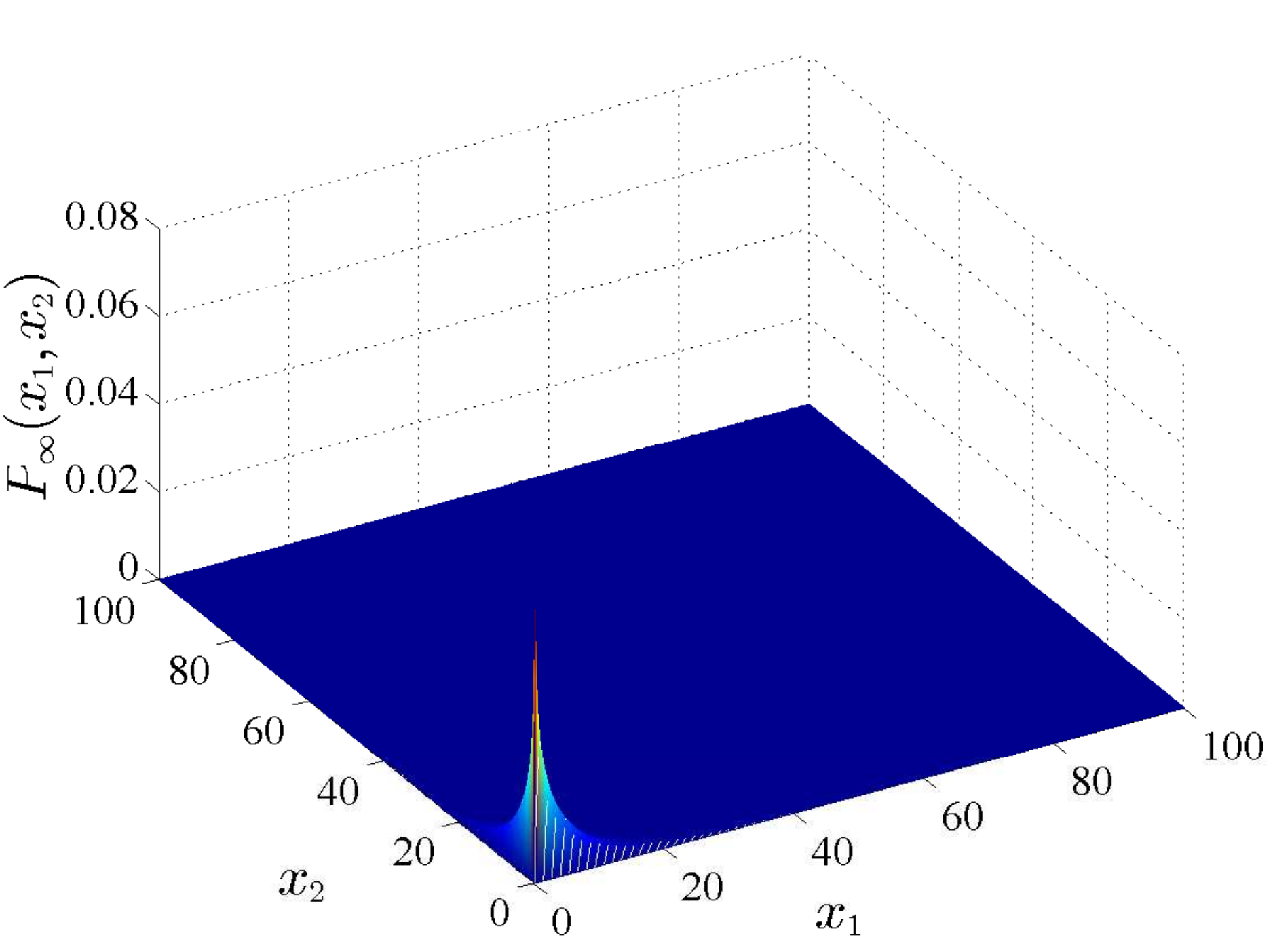}
	\end{minipage}
	\begin{minipage}[t]{0.49\linewidth}
		\centering
		\textbf{B}\\
		\includegraphics[width=1\textwidth]{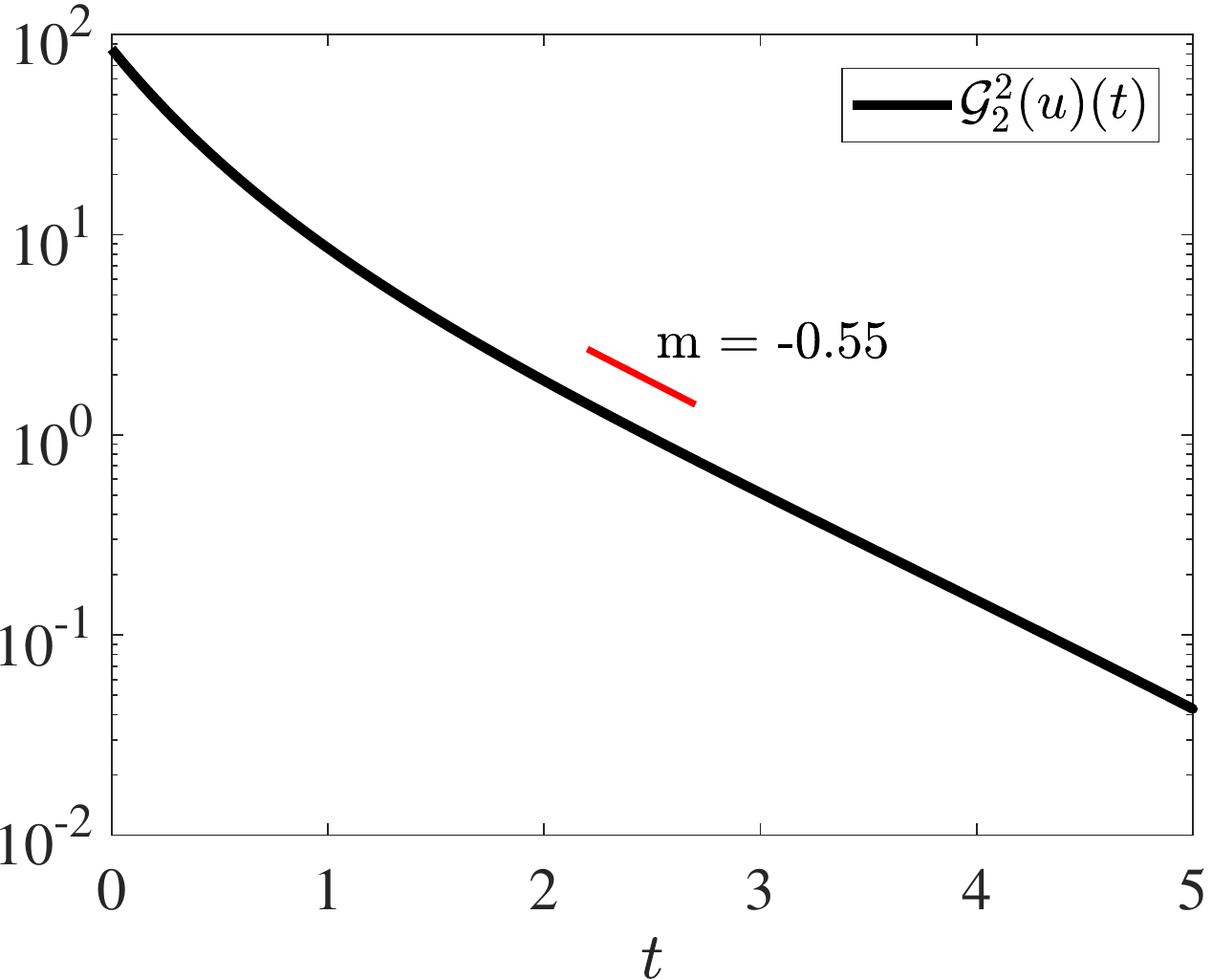}
	\end{minipage}
	\caption{Example of two self regulated proteins whose
          distribution has a peak in $\mathbf{x}=(0,0)$. (Same
          parameters as in the example depicted in Figure
          \ref{fig:case1} for both proteins)}
	\label{fig:EX2_2D}
\end{figure}
\begin{figure}[H]
	\centering
	\begin{minipage}[t]{0.49\linewidth}
		\centering
		\textbf{A}\\
		\includegraphics[width=1\textwidth]{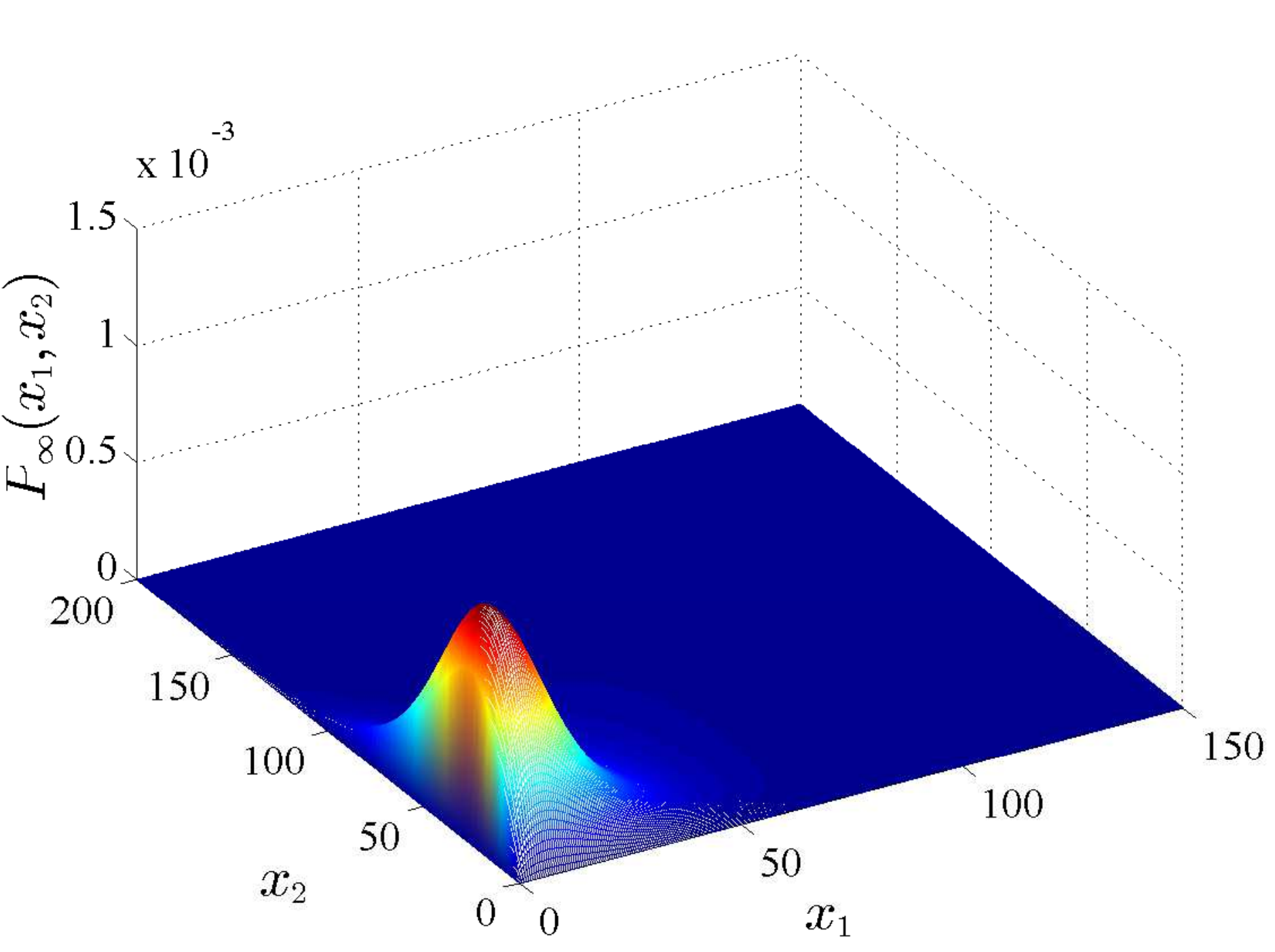}
	\end{minipage}
	\begin{minipage}[t]{0.49\linewidth}
		\centering
		\textbf{B}\\
		\includegraphics[width=1\textwidth]{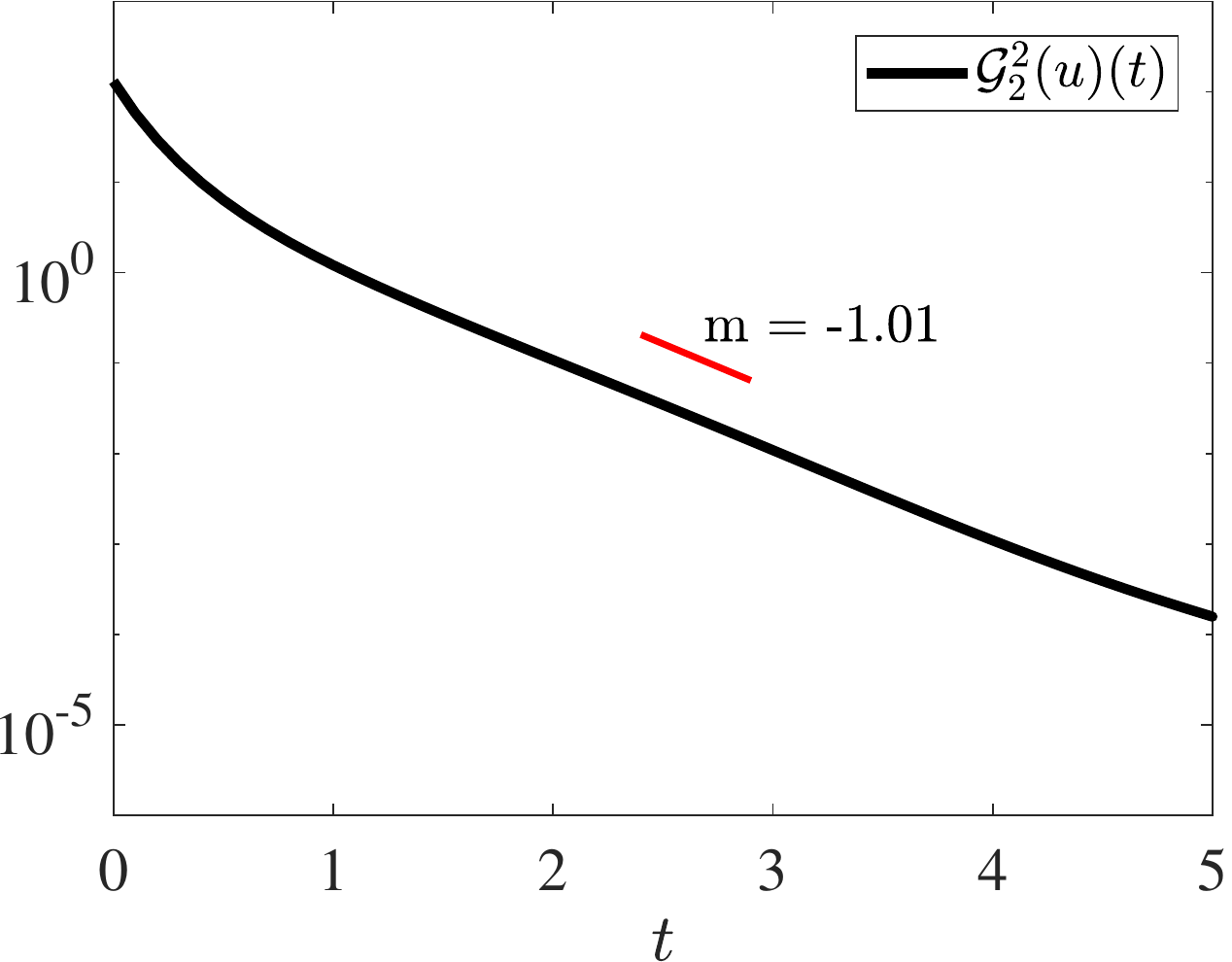}
	\end{minipage}
	\caption{Example of two self and cross regulated proteins
          whose distribution has a peak in some positive point
          $\mathbf{x}=(x_1,x_2)$ with $x_1>0$ and $x_2>0$. Parameters:
          $\gamma_x^1=\gamma_x^2=1$, $\gamma_m^1=\gamma_m^2=25$,
          $k_m^1=10$, $k_m^2=20$, $b_1=10$, $b_2=20$ and input
          functions in \eqref{eq:hill_multivariate}.}
	\label{fig:EX1_2D}
\end{figure}
\begin{figure}[H]
	\centering
	\begin{minipage}[t]{0.49\linewidth}
		\centering
		\textbf{A}\\
		\includegraphics[width=1\textwidth]{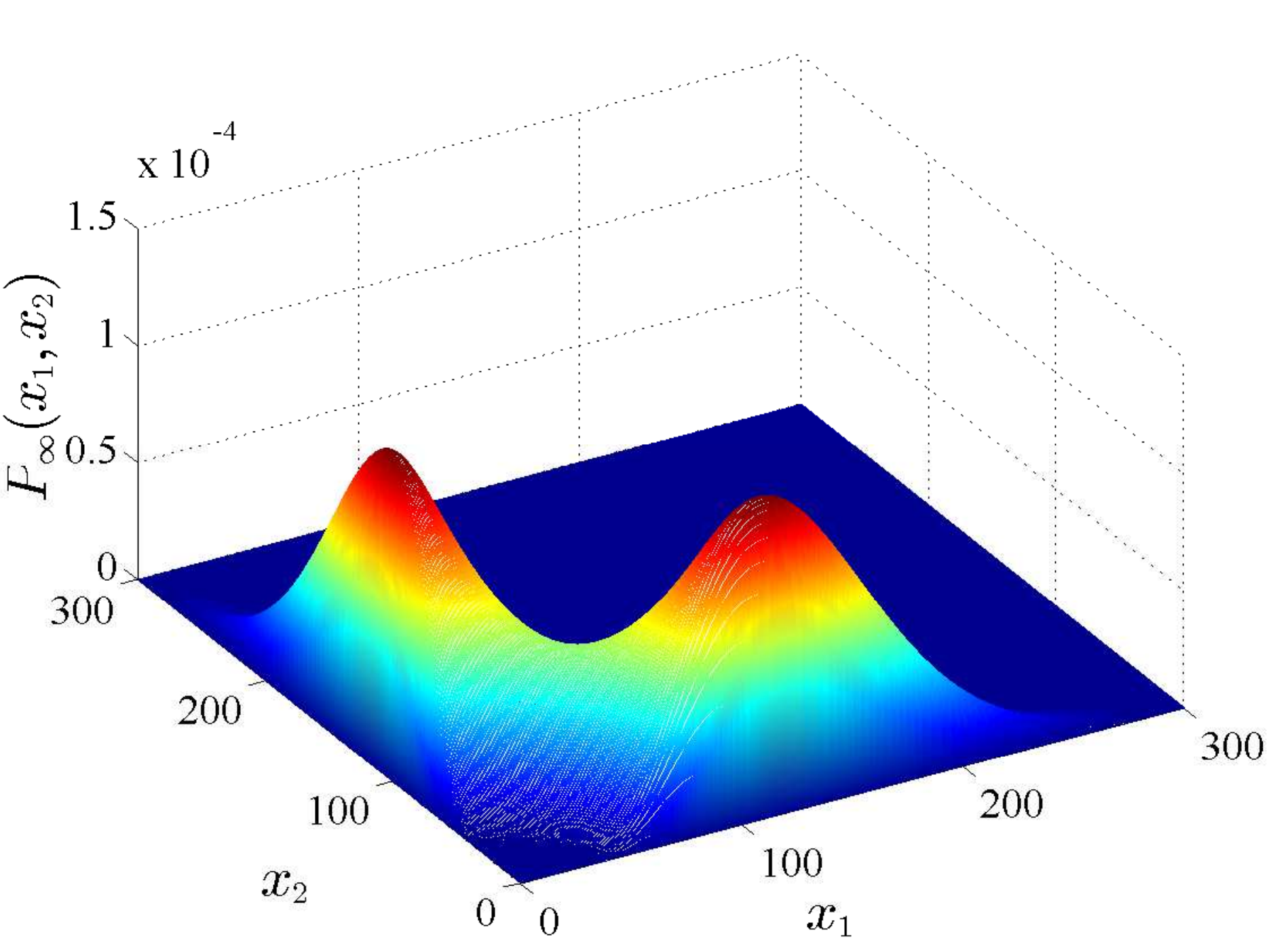}
	\end{minipage}
	\begin{minipage}[t]{0.49\linewidth}
		\centering
		\textbf{B}\\
		\includegraphics[width=1\textwidth]{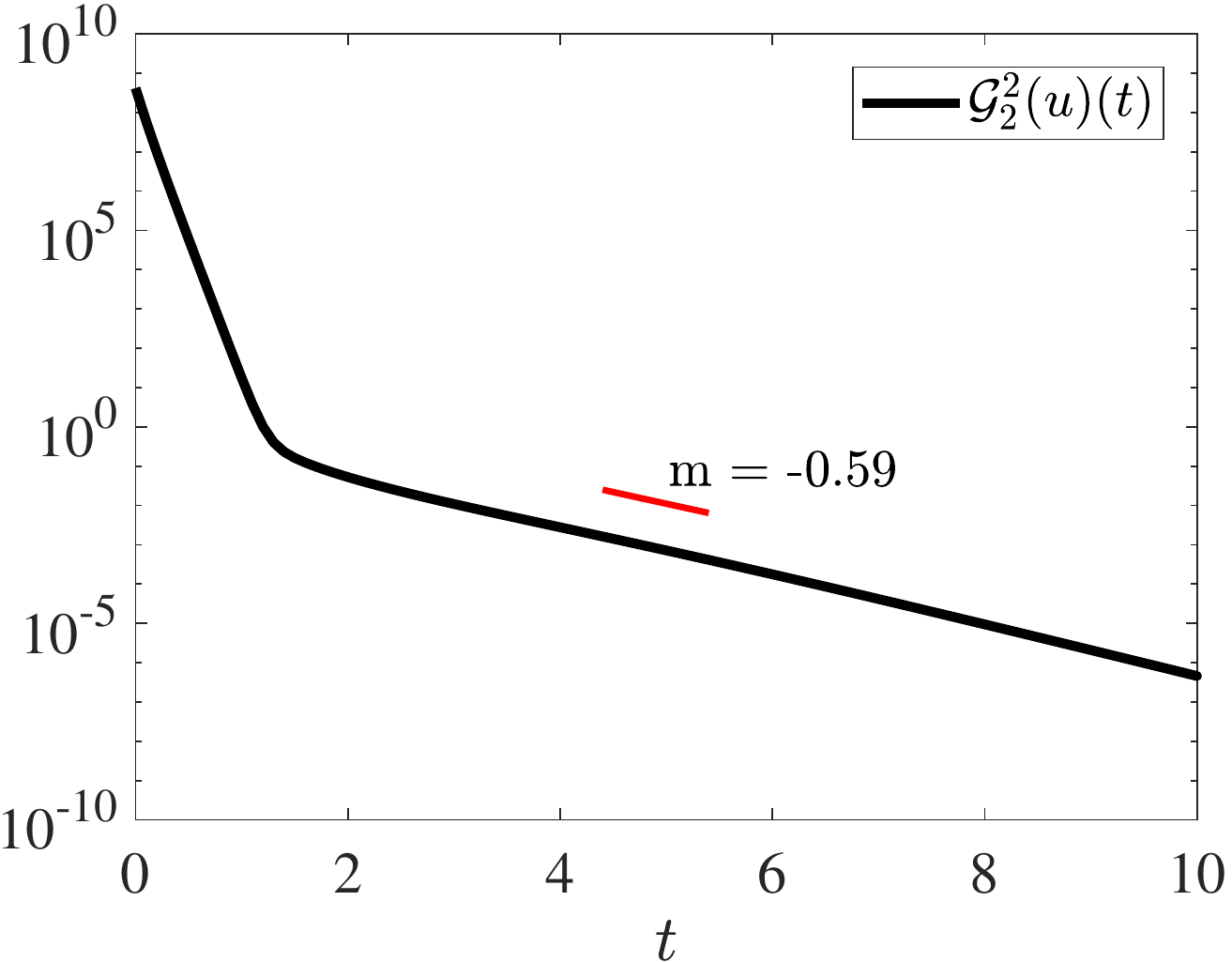}
	\end{minipage}
	\caption{Example of two  mutual repressed proteins whose joint distribution is bimodal attaining two peaks in two positive points. Parameters:$\gamma_x^1=\gamma_x^2=1$, $\gamma_m^1=\gamma_m^2=25$, $k_m^1=k_m^2=8$ and $b_1=b_2=16$ with input functions defined in \eqref{eq:hill_cross_reg}.}
	\label{fig:EX4_2D}
\end{figure}

\section*{Acknowledgements}

J.~A.~Cañizo and J.~A.~Carrillo were supported by projects
MTM2014-52056-P and MTM2017-85067-P, funded by the Spanish government
and the European Regional Development Fund. J.~A.~Carrillo was
partially supported by the EPSRC grant number EP/P031587/1. M. Pájaro acknowledges support from Spanish MINECO fellowships BES-2013-063112, EEBB-I-16-10540 and EEBB-I-17-12182.

\bibliographystyle{abbrv}

\bibliography{biblioentropia}  

\end{document}